\documentclass[a4paper,12pt]{amsart}
\usepackage{amssymb}
\usepackage{xypic}
\usepackage[all,cmtip]{xy}
\usepackage{amsmath}
\usepackage{chemarrow}
\usepackage[colorlinks, %注释掉此项则交叉引用为彩色边框(将colorlinks和pdfborder同时注释掉)
pdfborder=001, %注释掉此项则交叉引用为彩色边框
linkcolor=cyan,
anchorcolor=magenta%red,
citecolor=green
]{hyperref}
\usepackage{mathrsfs}
\usepackage{wasysym}
\usepackage{titletoc}
\usepackage{bm}
\usepackage{bbm}
\usepackage{bbold} %空心数字
\usepackage{cite}
\usepackage{subfiles}
\usepackage{extarrows} %长等号
\usepackage{geometry}
\usepackage{colordvi}
\usepackage{color}
\usepackage{stmaryrd}%\mapsfrom

\allowdisplaybreaks

\DeclareMathOperator{\End}{End}
\DeclareMathOperator{\Ext}{Ext}

\DeclareMathOperator{\gld}{gldim}

\DeclareMathOperator{\Hom}{Hom}

\DeclareMathOperator{\id}{id}
\DeclareMathOperator{\injdim}{injdim}

\DeclareMathOperator{\Kdim}{Kdim}

\DeclareMathOperator{\kk}{\Bbbk}
\DeclareMathOperator{\lgld}{lgld}
\DeclareMathOperator{\lhu}{\leftharpoonup}

\DeclareMathOperator{\M}{\mathcal{M}}

\DeclareMathOperator{\N}{\mathbb{N}}

\DeclareMathOperator{\pd}{pd}

\DeclareMathOperator{\rgld}{rgld}

\DeclareMathOperator{\rhu}{\rightharpoonup}

\DeclareMathOperator{\Tor}{Tor}

\DeclareMathOperator{\vep}{\varepsilon}

\DeclareMathOperator{\vtr}{\vartriangleright}

\DeclareMathOperator{\Z}{\mathbb{Z}}

\newcommand{\To}{\longrightarrow}

\numberwithin{equation}{section}

\theoremstyle{definition}

\newtheorem{thm}{Theorem}[section]
\newtheorem{prop}[thm]{Proposition}
\newtheorem{lem}[thm]{Lemma}

\newtheorem{cor}[thm]{Corollary}
\newtheorem{defn}[thm]{Definition}

\newtheorem{ques}[thm]{Question}
\newtheorem{ex}[thm]{Example}
\newtheorem{hypo}[thm]{Hypothesis}
\newtheorem{conj}[thm]{Conjecture}

\begin{document}

\title{Artin-Schelter Gorenstein property of Hopf Galois extensions}%{Skew Calabi-Yau property of Hopf Galois extensions}

%\author{Quanshui Wu}
%\address{School of Mathematical Sciences, Fudan University, Shanghai 200433, China}
%\email{qswu@fudan.edu.cn}

\author{Ruipeng Zhu}
\address{School of Mathematics, Shanghai University of Finance and Economics, Shanghai 200433, China}
\email{zhuruipeng@sufe.edu.cn}

\begin{abstract}
%	We study homological properties of faithfully flat Hopf Galois extension $A \subseteq B$ and prove that $B$ is AS Gorenstein if $B$ is a noetherian affine PI algebra and $A$ is also AS Gorenstein.
%	We show that injective dimension is a monoidal invariant for AS Gorenstein Hopf algebras, in the sense that if two AS Gorenstein Hopf algebras have equivariant monoidal categories of comodules, then their injective dimension should be equal.
	This paper investigates the homological properties of the faithfully flat Hopf Galois extension $A \subseteq B$.
	It establishes that when $B$ is a noetherian affine PI algebra and $A$ is AS Gorenstein, $B$ inherits the AS Gorenstein property.
	Furthermore, we demonstrate that injective dimension serves as a monoidal invariant for AS Gorenstein Hopf algebras. Specifically, if two such Hopf algebras have equivalent monoidal categories of comodules, then their injective dimensions are equal.
\end{abstract}

%{Artin-Schelter Gorenstein algebra, Hopf Galois extension, Ehresmann--Schauenburg bialgebroid, monoidal Morita--Takeuchi equivalence}
%{Skew Calabi-Yau algebra, Artin-Schelter regular algebra, Nakayama automorphism, Hopf Galois extension}

%\thanks{}

\maketitle

%\titlecontents{section}[0pt]{\addvspace{2pt}\filright}
%{\contentspush{\thecontentslabel\ }}
%{}{\titlerule*[8pt]{.}\contentspage}

%\tableofcontents
%\setlength{\baselineskip}{1.4em}
%\textwidth=125mm
%\textheight=195mm

\section*{Introduction}

Over two decades ago, Brown and Goodearl embarked on exploring ring-theoretic attributes of infinite-dimensional Hopf algebras, with their findings detailed in the literature \cite{Bro1998, BG1997}.
Their investigation revealed that numerous examples of noetherian affine Hopf algebras possess the Gorenstein property. Furthermore, they demonstrated that noetherian affine PI, Gorenstein Hopf algebras exhibit various favorable homological characteristics.

In \cite{Bro1998}, Brown posed a series of inquiries pertaining to general noetherian (PI) Hopf algebras, among which is the following noteworthy question \cite[Question A]{Bro1998}:

\begin{conj}[Brown-Goodearl conjecture]\label{BG-conj}
	Is every noetherian Hopf algebra AS Gorenstein?
\end{conj}

Since then, extensive research has explored the homological characteristics associated with the AS Gorenstein property, including the existence of rigid dualizing complexes for infinite-dimensional noetherian Hopf algebras.
More recently, the Brown–Goodearl conjecture has been explored in the context of other classes of noetherian algebras that exhibit similarities to Hopf algebras, such as weak Hopf algebras, as reported in \cite{RWZ2021}.
In this paper, we delve into the Brown–Goodearl question specifically for faithfully flat Hopf Galois extensions, which can be regarded as noncommutative analogues of principal bundles.

Applying a result of Wu and Zhang \cite{WZ2003}, we have the following result concerning whenever faithfully flat Hopf Galois extension of  an AS Gorenstein algebra is also AS Gorenstein, see Theorem \ref{PI-Hopf-Galois-ext-is-AS-Gorenstein}.

\begin{thm}
	Let $H$ be a Hopf algebra, and $B$ be a noetherian affine PI algebra which is a faithfully flat $H$-Galois extension of $A := B^{co H}$.
	If $A$ is AS Gorenstein, then $B$ is as well.
	Additionally, if $H$ is AS Gorenstein of dimension $d_H$, then $\injdim({B_B}) = \injdim({A_A}) + d_H$. %is AS Gorenstein of dimension $d_A + d_H$ where $d_A$ is the injective dimension of $A$.
	%	If $A$ and $H$ are both noetherian AS Gorenstein, then $B$ is also AS Gorenstein of dimension $d_A + d_H$.
\end{thm}

%Let's consider a finiteness condition in the following, which is useful for group algebras \cite[Chapter VIII]{Bro1982}.
%In the subsequent discussion, we shall consider a finiteness condition that proves to be beneficial for group algebras.
In the subsequent discussion, we delve into a specific finiteness condition that offers valuable insights for group algebras.

Let $R$ be a ring.
An $R$-module $M$ is said to be of type $\mathbf{FP}_n$ ($n \geq 0$) if there is a projective resolution
\[ \cdots \To P_{n+1} \To P_n \To \cdots \To P_1 \To P_0 \To M \To 0, \]
where $P_0, P_1, \dots, P_n$ are finitely generated.
Furthermore, $M$ is said to be of type $\mathbf{FP}_{\infty}$ if there is a projective resolution
\[ \cdots \To P_{n+1} \To P_n \To \cdots \To P_1 \To P_0 \To M \To 0, \]
where each $P_n$ is finitely generated for all integers $n \geq 0$.
An $R$-module $M$ of type $\mathbf{FP_{\infty}}$ is said to be of type $\mathbf{FP}$, if $M$ has finite projective dimension.

Let $H$ be a Hopf algebra over a field $\kk$.
Then $\kk$ is a left $H$-module induced by the counit map $\vep: H \to \kk$, which is called the trivial left module of $H$.
We will say that $H$ is of type $\mathbf{FP}_n$ ($n \in \N \cup \{ \infty, \emptyset \} $) if the trivial left module $\kk$ is of type $\mathbf{FP}_n$ as an $H$-module.
%There are plenty of examples of group algebra of type $\mathbf{FP_{\infty}}$, see \cite[Chapter VIII]{Bro1982} for example.
%Clearly, every noetherian Hopf algebra is always of type $\mathbf{FP_{\infty}}$.
%Notably, there exist numerous examples of group algebras that belong to the type $\mathbf{FP_{\infty}}$, as discussed in \cite[Chapter VIII]{Bro1982}.
Evidently, noetherian Hopf algebras are always of type $\mathbf{FP_{\infty}}$. However, it is worth mentioning that there are numerous non-noetherian group algebras that also belong to the type $\mathbf{FP_{\infty}}$, as explored in Chapter VIII of \cite{Bro1982}.

In \cite{Sch1996}, Schauenburg %\cite[Corollary 5.7]{Sch1996} 
has shown the equivalence of the following assertions:
\begin{enumerate}
	\item[(1)] $L$ and $H$ are {\it monoidally Morita--Takeuchi equivalent}, that is, their comodule categories are monoidally equivalent;
	\item[(2)] there is a $L$-$H$-biGalois extension of $\kk$.
\end{enumerate}

%Recently, some homological propreties, such as global dimension, skew Calabi -Yau property have been study broadly \cite{Bic2016, WYZ2017, Bic2022} under monoidal Morita--Takeuchi equivalence, see the following results and questions.
Recently, several homological properties, including global dimension and skew Calabi-Yau property, have garnered significant attention and have been extensively studied in the context of monoidal Morita--Takeuchi equivalence, see \cite{Bic2016, WYZ2017, Bic2022}.
This research area has produced many fascinating results and questions, some of which are outlined below.

\begin{thm}\label{mMTE-gld-thm}
	Let $H$ and $L$ be two monoidally Morita--Takeuchi equivalent Hopf algebras of type $\mathbf{FP}$.
	
	(1) \cite[Theorem 8]{Bic2022} Then they have the same global dimension, that is, $\gld (H) = \gld (L)$.
	
	(2) \cite[Theorem 2]{WYZ2017} If $H$ is AS regular of dimension $d$, so is $L$. %then $L$ is AS regular of dimension $d$ as well.
\end{thm}

\begin{ques}\label{ques-0}
	Let $H$ and $L$ be two monoidally Morita--Takeuchi equivalent Hopf algebras.
	
	(1) \cite{Bic2016, Bic2022} Then $\gld(H) = \gld(L)$?
	
	(2)\cite[Question 3]{WYZ2017} Suppose that $H$ is of type $\mathbf{FP}$. Is $L$ always of type $\mathbf{FP}$?
\end{ques}

%The article \cite{Bic2022} can be regarded as a fundamental reference for the Qusetion \ref{ques-0}.
%Although the above qusetion is true for numerous examples, there is a counterexample, see Example \ref{ex}. 
%%Although numerous examples hold true, we have a counterexample, see Example \ref{ex}.
%Instead, we will consider which alternative homological conditions might be monoidal invariants.
%Type $\mathbf{FP}_{\infty}$ is proved to be a monoidal invariant for Hopf algebras in Theorem \ref{mMT-preserve-FP_infty}.
The article \cite{Bic2022} serves as a fundamental reference for Question \ref{ques-0}. While the question holds true for a significant number of examples, there exists a counterexample, as demonstrated in Example \ref{ex}. Despite this, we will explore alternative homological conditions that might exhibit monoidal invariance. Notably, Theorem \ref{mMT-preserve-FP_infty} will prove that being of type $\mathbf{FP}_{\infty}$ is a monoidal invariant for Hopf algebras. This theorem suggests that even in the presence of counterexamples, certain homological properties can still retain their invariance under monoidal Morita--Takeuchi equivalence, providing valuable insights into the structure and behavior of these algebras.
Moreover, we have established the following theorem.

\begin{thm}\label{main-thm-1}
	Let $H$ and $L$ be two Hopf algebras that are monoidally Morita--Takeuchi equivalent and of type $\mathbf{FP}_{\infty}$.
	Suppose that the antipodes of $H$ and $L$ are both bijective.
	
	(1) If $H$ has finite global dimension $d$, then the injective dimension $\injdim(L) = d$.
	
	(2) If $H$ is AS regular of dimension $d$, then $L$ is AS Gorenstein of dimension $d$.
\end{thm}
 
Based on this, we pose the following question.

\begin{ques}\label{main-ques}
	Let $H$ and $L$ be two Hopf algebras that are monoidally Morita--Takeuchi equivalent. %and of type $\mathbf{FP}_{\infty}$.
	Does it follows that $\injdim (L)  = \injdim (H)$?
\end{ques}

We are able to answer this question under the AS Gorenstein assumption.

\begin{thm}\label{main-thm-2}
	Let $H$ and $L$ be two Hopf algebras with bijective antipodes that are monoidally Morita--Takeuchi equivalent.
	If both $L$ and $H$ are AS Gorenstein of type $\mathbf{FP}_{\infty}$, then $\injdim (L)  = \injdim(H)$.
\end{thm}

Notice that any noetherian AS Gorenstein Hopf algebra has a bijective antipode, as established in \cite{LOWY2018}.
%The following is a immediate corollary of a celebrated result that any noetherian affine PI Hopf algebra is AS Gorenstein which is proved by Wu and Zhang \cite[Theorem 0.2]{WZ2003}.
The following corollary is a direct consequence of the famous result established by Wu and Zhang in \cite[Theorem 0.2]{WZ2003}, which states that any noetherian affine PI Hopf algebra is AS Gorenstein.

\begin{cor}\label{PI-Hopf-alg-mMT-preserve-injdim-cor}
	Let $L$ and $H$ be two noetherian affine PI Hopf algebras.
	If they are monoidally Morita--Takeuchi equivalent, then $\injdim (L)  = \injdim(H)$.
\end{cor}

%Now it becomes clear that an affirmative answer to the Brown-Goodearl conjecture (Conjecture \ref{BG-conj}) will help to answer Question \ref{main-ques} in the case that both $L$ and $H$ are noetherian.
%It is now evident that a positive response to the Brown-Goodearl conjecture (Conjecture \ref{BG-conj}) would significantly contribute to addressing Question \ref{main-ques} in scenarios where both L and H possess the noetherian property.
A confirmation of the Brown-Goodearl conjecture (Conjecture \ref{BG-conj}) would significantly advance the resolution of Question \ref{main-ques}, especially when both $L$ and $H$ are noetherian.

%%A few years ago Brown and Goodearl started to study ring-theoretic properties of a class of infinite dimensional Hopf algebras \cite{Bro1998, BG1997}.
%%They verified that many examples of noetherian affine Hopf algebras are Gorenstein, and showed that noetherian affine PI, Gorenstein, Hopf algebras have various good homological properties.
%%In \cite{Bro1998}, Brown posted a list of questions concerning general noetherian PI Hopf algebras, and one of them is the following \cite[Question A]{Bro1998}: 
%
%\begin{conj}[Brown-Goodearl conjecture]\label{BG-conj}
%	Is every noetherian Hopf algebra AS Gorenstein?
%\end{conj}

%The paper is organized as follows.
%
%We begin, in Section 1, by reviewing the definitions and basic properties of Hopf Galois extension and studying homological dimensions by using the Stefan's spectral sequence.
%In Section 2, we prove that faithfully flat Hopf Galois extension of an AS Gorenstein algebra is also AS Gorenstein.
%In Section 3, we study the injective dimension of Ehresmann--Schauenburg bialgebroids which allow us, in Section 4, to prove that $H$ is an AS Gorenstein algebra of dimension $d$ if there is a faithfully flat $H$ Galois extension of the base field $\kk$ which is skew Calabi-Yau of dimension $d$.
%In Section 4, we study the consequences of this result, show the Theorems \ref{main-thm-1} and \ref{main-thm-2}, and give a counterexample for the Qusetion \ref{ques-0}.

The paper is structured as follows.
In Section 1, we introduce the fundamental definitions and characteristics of Hopf Galois extensions. Additionally, we delve into the investigation of homological dimensions by utilizing Stefan's spectral sequence.
Section 2 focuses on establishing that a faithfully flat Hopf Galois extension of an AS Gorenstein algebra retains the AS Gorenstein property.
Section 3 examines the injective dimension of Ehresmann--Schauenburg bialgebroids.
%This exploration paves the way for Section 4, where we demonstrate that an Hopf algebra $H$ is AS Gorenstein of dimension $d$ if it possesses a faithfully flat $H$-Galois extension of the base field $\kk$ that exhibits skew Calabi-Yau properties of dimension $d$.
This investigation sets the stage for Section 4, where we show that a Hopf algebra $H$ is AS Gorenstein of dimension $d$ if it admits a faithfully flat $H$-Galois extension $B$ over the base field $\kk$, with $B$ being a skew Calabi-Yau algebra of dimension $d$.
%In Section 4, we further explore the implications of this finding, highlighting Theorems \ref{main-thm-1} and \ref{main-thm-2}. Additionally, we provide a counterexample to address Question \ref{ques-0}.
In Section 4, we also elaborate on the implications of our results, focusing particularly on Theorems \ref{main-thm-1} and \ref{main-thm-2}. Furthermore, we offer a counterexample to respond to Question \ref{ques-0}.

\section{Preliminaries}

Throughout this paper, $\kk$ is a base field, and all vector spaces and algebras are over $\kk$. Unadorned $\otimes$ means $\otimes_{\kk}$ and $\Hom$ means $\Hom_{\kk}$.
Let $R$ be an algebra over the base field $\kk$. Let $R^{op}$ denote the opposite ring of $R$, and let $R^e$ denote the enveloping algebra $R \otimes R^{op}$.
%Note that the flip map $a \otimes b \mapsto b \otimes a$ is an anti-automorphism of the algebra $R^e$.
Observe that the flip map, which sends $a \otimes b$ to $b \otimes a$, constitutes an anti-automorphism of the algebra $R^e$.
Usually we work with left modules.
A right $R$-module is viewed as an $R^{op}$-module, and an $R$-$R$-bimodule is the same as an $R^e$-module.
%An $A$-module is called {\it finite} if it is finitely generated over $A$.
%In this paper, $\kk$ is a field of characteristic zero, and 
%$H$ will stand for a Hopf algebra $(H, \Delta, \varepsilon)$ with a bijective antipode $S$ over a field $\kk$. We use the Sweedler notation $\Delta(h) = \sum_{(h)} h_1 \otimes h_2$ for all $h \in H$.
%%For the definitions related to Hopf algebras we refer to \cite{Mon}.
%We recommend \cite{Mon1993} as a basic reference for the theory of Hopf algebras.

Let $M$ be a left module over an algebra $R$.
Let $\pd_R(M)$ denote the projective dimension of $M$ and $\injdim({_RM})$ the injective dimension of $M$.
Let $\lgld(R)$ (resp., $\rgld(R)$) denote the left (resp., right) global dimension of $R$.
We say that $R$ has finite injective dimension if the modules $_RR$ and $R_R$ have finite injective dimensions which are equal, and denote these numbers by $\injdim(R)$.
We say that $R$ has finite global dimension if the left and right global dimensions of $R$ are finite and equal, and denote these numbers by $\gld(R)$.

Let $H$ stand for a Hopf algebra $(H, \Delta, \vep)$ with antipode $S$.
We use the Sweedler notation $\Delta(h) = \sum_{(h)} h_1 \otimes h_2$ for all $h \in H$.
We recommend \cite{Mon1993} as a basic reference for the theory of Hopf algebras.
%In our subsequent considerations, the antipode of the given Hopf algebra is always assumed to be bijective.

%The antipode of a finite-dimensional Hopf algebra is bijective \cite{LS1969}.
%Skryabin proved that the antipode of a noetherian Hopf algebra is injective and conjectured that it is bijective \cite{Skr2006}. 
%Le Meur proved that the antipode of a Hopf algebra that has Van den Bergh duality is bijective \cite{LMeu2019}.
%And any noetherian AS Gorenstein Hopf algebra always has a bijective antipode \cite{LOWY2018}.
%There are examples of Hopf algebras whose antipodes are not bijective, see \cite{Tak1971} or \cite{Sch2000a}.
The antipode of a finite-dimensional Hopf algebra is bijective \cite{LS1969}.
Skryabin showed that the antipode of a noetherian Hopf algebra is injective and hypothesized it to be bijective \cite{Skr2006}.
Le Meur confirmed the bijectivity of the antipode for Hopf algebras with Van den Bergh duality \cite{LMeu2019}.
Furthermore, noetherian AS Gorenstein Hopf algebras are guaranteed to have a bijective antipode \cite{LOWY2018}.
Nonetheless, there exist Hopf algebras with non-bijective antipodes, as seen in \cite{Tak1971} and \cite{Sch2000a}.
Let $H$ be a Hopf algebra with a bijective antipode $S$.
Since the antipode $S: H \to H^{op}$ is an algebra isomorphism, left and right injective dimensions of $H$ coincide.
%In the rest of this paper, we always assume that the antipode of any Hopf algebra is bijective.
In the rest of this paper, the antipode of the given Hopf algebra is always assumed to be bijective.

\subsection{Faithfully flat Hopf Galois extensions}

Let $H$ be a Hopf algebra over the field $\kk$, and $B$ be a right $H$-comodule algebra via an algebra homomorphism $\rho_B: B \to B \otimes H$. The {\it coinvariant subalgebra} of the $H$-coaction on $B$ is $\{ b \in B \mid \rho_B(b) = b \otimes 1 \}$, denoted by $B^{co H}$.

\begin{defn}
	The extension $A(:= B^{co H}) \subset B$ is called (right) $H$-{\it Galois} if the canonical map
	\begin{center}
		$\beta: B \otimes_A B \to B \otimes H, \qquad b' \otimes_A b \mapsto b'\rho(b) = \sum\limits_{(b)} b'b_0 \otimes b_1$
	\end{center}
	is bijective.
\end{defn}

Let $R$ be an algebra, and $_R\!\M$ (resp., $\M_R$) be the category of left (resp., right) $R$-modules.
If $R$ is a right $H$-comodule algebra via $\rho_R: R \to R \otimes H$, then a relative right-right $(R, H)$-Hopf module is a vector space with a right $R$-action and a right $H$-coaction $\rho_M$ such that
\[ \rho_M(mr) = \rho_M(m)\rho_R(r) = \sum_{(m), (r)} m_0r_0 \otimes m_1r_1 \]
for all $m \in M$ and $r \in R$.
%The category of relative right $(B, H)$-Hopf modules $\M^H_B$ has direct sums, and is a Gorendieck category( with enough injective objects).
A relative left-right $(R, H)$-Hopf module is defined similarly.
The category of relative right-right (resp., left-right) $(R, H)$-Hopf modules is denoted by $\M^H_R$ (resp., $_R\!\M^H$).

We require some known facts about Hopf Galois extensions.

\begin{lem}\cite[(1.6)]{Doi1985}%这个文章找不到
	\cite[Remark 3.3]{Sch1990}\label{Hopf-module-injective}
	Let $B$ be a right $H$-comodule algebra. Then the following are equivalent.
	
	(1) $B$ is an injective right $H$-comodule.
	
	(2) There is a right $H$-colinear and unitary map $\varphi: H \to B$.
	
	(3) Any Hopf module in $\M^H_B$ is injective as right $H$-comodule.
%	If the antipode of $H$ is bijective, then each of the above is equivalent to:
%	\begin{enumerate}
%		\item[(4)] Any Hopf module in $_B\M^H$ is injective as right $H$-comodule.
%	\end{enumerate}
\end{lem}

\begin{thm}\cite[Theorem I]{Sch1990}\label{faithful-flat-Hopf-Galois-extension-thm} %\cite[Theorem 5.6]{SS2005}
	Let $H$ be a Hopf algebra with a bijective antipode, $B$ be a right $H$-comodule algebra and $A := B^{co H}$. Then the following are equivalent.
	
	(1) $B$ is injective as right $H$-comodule, and $\beta: B \otimes_A B \to B \otimes H$ is surjective.
	
	(2) $(- \otimes_A B, (-)^{co H})$ is an equivalence between $\M_A$ and $\M^H_B$. %$\M_A \to \M^H_B, \; M \mapsto M \otimes_A B$ is an equivalence.
	
	(3) $(B \otimes_A -, (-)^{co H})$ is an equivalence between $_A\!\M$ and $_B\!\M^H$. %${_A\M} \to {_B\M^H}, \; M \mapsto B \otimes_A M$ is an equivalence.
	
	(4) $B$ is a faithfully flat left $A$-module, and $A \subseteq B$ is an $H$-Galois extension.%$\beta$ is an isomorphism.
	
	(5) $B$ is a faithfully flat right $A$-module, and $A \subseteq B$ is an $H$-Galois extension.%$\beta$ is an isomorphism.
	%		\item \begin{enumerate}
		%			\item $B$ is a direct summand as left $A$-module.
		%			\item $A \subseteq B$ is an $H$-Galois extension.%$\beta$ is an isomorphism.
		%		\end{enumerate}
	%		\item \begin{enumerate}
		%			\item $B$ is a direct summand as right $A$-module.
		%			\item $A \subseteq B$ is an $H$-Galois extension.%$\beta$ is an isomorphism.
		%		\end{enumerate}
	
%	(6) $A \subseteq B$ is a Hopf Galois extension, and the comodule algebra $B$ is right $H$-equivariantly projective as a left $A$-module, i.e., there exists a right $H$-colinear and left $A$-linear splitting of the multiplication $A \otimes B \to B$.
\end{thm}

Since conditions (4) and (5) in Theorem \ref{faithful-flat-Hopf-Galois-extension-thm} are equivalent, we say a Hopf Galois extension $A \subseteq B$ is faithfully flat if $B$ is faithfully flat as a left or right $A$-module.
%A faithfully flat $H$-Galois extension is also called a {\it principal homogeneous space} in \cite{Sch1990} or {\it principal $H$-extension} in \cite{BH2004}.

Let $A(:= B^{co H}) \subset B$ be an $H$-Galois extension.
For any $A$-$A$-bimodule $M$, we set 
\[ M^A := \{ m \in M \mid ma = am, \forall \, a \in A \}, \qquad M_A := M/[M, A] \]
where $[M, A]$ is the subspace of $M$ generated by all commutators $[m, a] := ma - am$.

%Then the restriction map of $\beta$ is a $\kk$-linear isomorphism from $(B \otimes_A B)^A$ to $B^A \otimes H$.
%Therefore we can define a $\kk$-linear map $\kappa: H \to (B \otimes_A B)^A, \; h \mapsto \beta^{-1}(1 \otimes h)$.
Following \cite{Brz1996} we shall call the $\kk$-linear map
\[ \kappa: H \to (B \otimes_A B)^A, \; h \mapsto \beta^{-1}(1 \otimes h) \]
the {\it translation map} associated to the $H$-Galois extension $A \subseteq B$.
For any $h \in H$ we shall use the notation
$$\kappa(h) = \sum \kappa^1(h) \otimes_A \kappa^2(h).$$

%Some useful properties of $\kappa$ in \cite{Sch1990} are listed in the following:
The following are key properties of $\kappa$ as outlined in \cite{Sch1990a}:
\begin{equation}\label{kappa(hk)}
	\sum\kappa^1(hk) \otimes_A \kappa^2(hk) = \sum \kappa^1(k)\kappa^1(h) \otimes_A \kappa^2(h)\kappa^2(k),
\end{equation}
\begin{equation}\label{akappa(h)=kappa(h)a}
	\sum a\kappa^1(h) \otimes_A \kappa^2(h) = \sum \kappa^1(h) \otimes_A \kappa^2(h)a,
\end{equation}
\begin{equation}\label{kappa^1-kappa^2_0-kappa^2_1}
	\sum\limits_{(h)} \kappa^1(h) \otimes_A \kappa^2(h)_0 \otimes \kappa^2(h)_1 = \sum\limits_{(h)} \kappa^1(h_1) \otimes_A \kappa^2(h_1) \otimes h_2,
\end{equation}
\begin{equation}\label{kappa^1_0-kappa^2-kappa^1_1}
	\sum\limits_{(h)} \kappa^1(h)_0 \otimes_A \kappa^2(h) \otimes \kappa^1(h)_1 = \sum\limits_{(h)} \kappa^1(h_2) \otimes_A \kappa^2(h_2) \otimes Sh_1,
\end{equation}
%\begin{equation}\label{b_0-kappa(b_1)}
%	\sum\limits_{(b)} b_0 \kappa^1(b_1) \otimes_A \kappa^2(b_1) = 1 \otimes_A b,
%\end{equation}
\begin{equation}\label{m-kappa=vep}
	\sum \kappa^1(h)\kappa^2(h) = \vep(h),
\end{equation}
%\begin{equation}\label{beta^{-1}}
%	\beta^{-1}(b \otimes h)  = \sum b\kappa^1(h) \otimes_A \kappa^2(h),
%\end{equation}
%\begin{equation}
%	\sum\limits_{(h)} \kappa^1(h_1) \otimes_A \kappa^2(h_1)\kappa^1(h_2) \otimes_A \kappa^2(h_2) = \sum \kappa^1(h) \otimes_A 1 \otimes_A \kappa^2(h),
%\end{equation}
%\begin{equation}
%	\sum\limits_{(y)} \kappa^1(Sy_1) \otimes_A y_0\kappa^2(Sy_1) \otimes y_2 = \sum\limits_{(y)} \kappa^1(Sy_1)y_0\kappa^2(Sy_1) \otimes_A 1 \otimes y_2,
%\end{equation}

Suppose that $B^{coH} = A$.
Then for any $B$-$B$-bimodule $M$, $M^A$ is a right $H$-module which is defined by
\begin{eqnarray}\label{H-mod-on-M^A}
	x \lhu h = \sum\limits_{(h)} \kappa^1(h) x \kappa^2(h)
\end{eqnarray}
for any $x \in M^A$ and $h \in H$;
and $M_A$ is a left $H$-module which is defined by
$$h \rhu \big( x +  [A, M] \big) = \sum\limits_{(h)} \kappa^2(h) x \kappa^1(h) + [A, M]$$
%$$h \rhu \overline{x} = \sum\limits_{(h)} \overline{\kappa^2(h) x \kappa^1(h)}$$
for any $x \in M$ and $h \in H$, see \cite[2.1, 2.2]{JS2006}.
%Both structures have already appeared in \cite{DT1989, Ste1995}, where they are called the Ulbrich-Miyashita actions.
%It is worth mentioning that the Ulbrich-Miyashita actions are natural, that is, for any $B^e$-modules $M$, $N$ and map $f \in \Hom_{B^e}(M, N)$, the induced maps
%$$f^0: M^A \To N^A, \; \text{ and } \; f_0: M_A \To N_A$$
%are both $H$-linear.
%Both these structures have been previously introduced and studied in \cite{DT1989, Ste1995}, where they are referred to as Ulbrich-Miyashita actions.
%It is noteworthy that these Ulbrich-Miyashita actions are functorial, in other words, we have two functors:
%\[ (-)^A: {_B\!\M_B} \To {\M_H}, \quad (-)_A: {_B\!\M_B} \To {_H\!\M}. \]
In prior literature, specifically \cite{DT1989} and \cite{Ste1995}, both these structures have been introduced and meticulously examined, where they are recognized as Ulbrich-Miyashita actions.
A crucial point to highlight is the functorial nature of these actions; in essence, this means that there exist two well-defined functors given by:
\[ (-)^A: {_B\!\M_B} \To {\M_H}, \quad (-)_A: {_B\!\M_B} \To {_H\!\M}. \]

\subsection{Stefan's spectral sequence for flat Hopf Galois extensions}

We say a Hopf Galois extension $A \subseteq B$ is flat if $B$ is flat both as a left and right $A$-module.
In this subsection, we give some results about homological properties of flat Hopf Galois extensions by using Stefan's spectral sequence.
%we recall the Stefan's spectral sequence and given some homological properties of flat Hopf Galois extensions.

%For any flat Hopf Galois extension $A \subseteq B$, Stefan \cite{Ste1995} defined a structure of right $H$-module on Hochschild cohomology group $\mathrm{HH}^i(A, M) (\cong \Ext^i_{A^e}(A, M))$ for any $B^e$-module $M$ by using universal $\delta$-functors, 
%%In \cite{Ste1995}, Stefan defined a structure of right $H$-module on $\Ext^i_{A^e}(A, M)$ for any $i \geq 0$,
%which is a generalization of the Miyashita-Ulbrich action.
%By the flatness of $_{A^e}B^e$, any injective $B^e$-module is also injective as an $A^e$-module.
%Hence the right $H$-module structure on $\Ext^i_{A^e}(A, M)$ is in fact obtained by taking homological groups of $\Hom_{A^e}(A, E^{\bullet})$ where $E^{\bullet}$ is a $B^e$-injective resolution of $M$.
%Then Stefan \cite{Ste1995} constructed a spectral sequence which link the Hochschild cohomologies of $A$ to those of $B$, where $A \subseteq B$ is a flat Hopf Galois extension.

For any given flat Hopf Galois extension $A \subseteq B$, Stefan in his work \cite{Ste1995} devised a right $H$-module structure on the Hochschild cohomology group $\mathrm{HH}^i(A, M)$, which is isomorphic to $\Ext^i_{A^e}(A, M)$, for any $B^e$-module $M$.
This construction was achieved by employing universal $\delta$-functors and serves as an extension of the Miyashita-Ulbrich action.

The inherent flatness property of $_{A^e}B^e$ ensures that every injective $B^e$-module also maintains its injectivity when considered as an $A^e$-module.
Consequently, the derived right $H$-module structure on $\Ext^i_{A^e}(A, M)$ can be effectively realized by computing homological groups from the complex $\Hom_{A^e}(A, E^{\bullet})$, where $E^{\bullet}$ is a $B^e$-injective resolution of $M$.

Subsequently, Stefan further developed in \cite{Ste1995} a spectral sequence that establishes a connection between the Hochschild cohomologies of $A$ and those of its flat Hopf Galois extension $B$.

\begin{thm}\cite[Theorem 3.3]{Ste1995}\label{H-Galois-spectral-sequence}
	Let $A \subseteq B$ be a flat $H$-Galois extension. For any $B^e$-module $M$, there is a convergent spectral sequence
	\[ \Ext^{p}_{H^{op}}(\kk, \Ext^q_{A^e}(A, M)) \Longrightarrow \Ext^{p+q}_{B^e}(B, M).\]
\end{thm}

Recall that $\pd (_H\kk) = \gld (H) = d$ by \cite[Corollary 1.4]{BG1997}.
Then we have the following immediate consequence.

\begin{cor}\label{pdim-B-as-B-B-bimodule}
	Let $A \subseteq B$ be a flat $H$-Galois extension. Then $\pd_{B^e}(B) \leq \pd(\kk_H) + \pd_{A^e}(A) = \gld(H) + \pd_{A^e}(A)$.
\end{cor}

In the following, we study the global dimension of faithfully flat Hopf Galois extension.

Let $M$ and $N$ be two modules over a $\kk$-algebra $R$.
Let $P_{\bullet}$ denote a projective resolution of $_{R^e}R$.
Given that the sequence $\cdots \to P_1 \to P_0 \to R \to 0$ is split exact as an $R$-module complex, the complex
\[\cdots \To P_1 \otimes_R M \To P_0 \otimes_R M \To R \otimes_R M (= M) \To 0\]
remains exact as well.
Owing to the projectivity of each module $P_i \otimes_R M$ for all $i$, the complex $P_{\bullet} \otimes_R M$ constitutes a projective resolution of the right $R$-module $M$.
Furthermore, utilizing the isomorphism
\[ \Hom_{R^e}(P_{\bullet}, \Hom(M, N)) \cong \Hom_{R^{op}}(P_{\bullet} \otimes_R M, N),\]
upon taking homology groups, we obtain the following isomorphism between Ext-groups,
\begin{equation}\label{RHom_R^e<->RHom_R}
	\Ext^i_{R^e}(R, \Hom(M, N)) \cong \Ext^i_{R^{op}}(M, N).
\end{equation}
for all $i \geq 0$.

\begin{lem}\label{H-Galois-spectral-sequenceII}
	Let $A \subseteq B$ be a flat $H$-Galois extension.
	For any right $B$-modules $M$ and $N$, there is a convergent spectral sequence
	\[ \Ext^{p}_{H^{op}}(\kk, \Ext^q_{A^{op}}(M, N)) \Longrightarrow \Ext^{p+q}_{B^{op}}(M, N).\]
\end{lem}
\begin{proof}
	In light of the isomorphism \eqref{RHom_R^e<->RHom_R}, for every integer $i \geq 0$, we have
	\[ \Ext^i_{A^e}(A, \Hom(M, N)) \cong \Ext^i_{A^{op}}(M, N).\]
	This implies that $\Ext^i_{A^{op}}(M, N)$ inherits a right $H$-module structure through this isomorphism.
	By Theorem \ref{H-Galois-spectral-sequence}, there exists a convergent spectral sequence which unfolds as follows:
	\begin{center}
		$\Ext^{p}_{H^{op}}(\kk, \Ext^q_{A^{op}}(M, N)) \Longrightarrow \Ext^{p+q}_{B^e}(B, \Hom(M, N)) \cong \Ext^{p+q}_{B^{op}}(M, N)$.
	\end{center}
	So we get the conclusion.
%	According to \eqref{RHom_R^e<->RHom_R}, for any $i \geq 0$, we have
%	\[ \Ext^i_{A^e}(A, \Hom(M, N)) \cong \Ext^i_{A^{op}}(M, N).\]
%	Hence $\Ext^i_{A^{op}}(M, N)$ is a right $H$-module which is induced by above isomorphism.
%	By Theorem \ref{H-Galois-spectral-sequence}, we have a convergent spectral sequence
%	\begin{center}
%		$\Ext^{p}_{H^{op}}(\kk, \Ext^q_{A^{op}}(M, N)) \Longrightarrow \Ext^{p+q}_{B^e}(B, \Hom(M, N)) \cong \Ext^{p+q}_{B^{op}}(M, N)$.
%	\end{center}
\end{proof}

The following result is an immediate consequence of Lemma \ref{H-Galois-spectral-sequenceII}.

\begin{prop}\label{Hopf-Galois-ext-prop-0}
	If $A \subseteq B$ is a flat $H$-Galois extension, then $\rgld(B) \leq \gld(H) + \rgld(A)$.
\end{prop}
%\begin{proof}
%	For any left $B$-modules $M$ and $N$, by Theorem \ref{H-Galois-spectral-sequence} and \eqref{RHom_R^e<->RHom_R}, we have
%	\begin{align*}
%		\RHom_B(M, N) \cong & \RHom_{B^e}(B, \Hom(M, N)) \\
%		\cong & \RHom_{H^{op}}(\kk, \RHom_{A^e}(A, \Hom(M, N))) \\
%		\cong & \RHom_{H^{op}}(\kk, \RHom_{A}(M, N)).
%	\end{align*}
%	So we get the conclusion.
%\end{proof}
The following example illustrates that within the above proposition, the equality may not be attained.

%\begin{ex}
%	Let $H$ be a finite dimensional Hopf algebra, and $H^*$ be the dual Hopf algebra of $H$. Then the Heisenberg double $B:= H \# H^* \cong \End(H)$ is a $H^*$-Galois extension of $A = H$. Hence $\rgld (B) \leq \pd_{B^e}(B) = 0$, that is, $B$ is semisimple. If $H$ is not cosemisimple, that is, $H^*$ is not semisimple, then $\pd_{B^e}(B) < \gld(H^*) + \pd(A_{A^e})$ and $\rgld(B) < \gld(H^*) + \rgld(A)$.
%\end{ex}

\begin{ex}
%	Consider a finite-dimensional Hopf algebra $H$, with its dual Hopf algebra $H^*$.
%	The Heisenberg double, denoted by $B:= H \# H^*$, can be isomorphic to $\End(H)$ which is separable algebra and constitutes an $H^*$-Galois extension of the subalgebra $A = H$.
%	Consequently, the right global dimension of $B$ satisfies $\rgld (B) \leq \pd_{B^e}(B)$, which in this case equals zero, thus indicating that $B$ is semisimple. If $H$ lacks cosemisimplicity, implying that $H^*$ is not semisimple, then we observe the following strict inequalities:
	For a finite-dimensional Hopf algebra $H$ and its dual $H^*$, the Heisenberg double $B:= H \# H^*$ is isomorphic to $\End(H)$, a separable algebra that forms an $H^*$-Galois extension over its subalgebra $A = H$. When $H$ is not cosemisimple, meaning $H^*$ is non-semisimple, the following strict inequalities hold:
	$\pd_{B^e}(B) < \gld(H^*) + \pd_{A^e}(A)$ and $\rgld(B) < \gld(H^*) + \rgld(A)$.
\end{ex}

\section{Faithfully flat PI Hopf Galois extensions are always AS Gorenstein}

\subsection{Artin-Schelter Gorenstein Hopf algebras}

%Recall that an algebra is Gorenstein if it has finite left and right injective dimension over itself.
%Let's consider the following stronger variant of the Gorenstein property for noncommutative rings.
Recall that an algebra is called Gorenstein when it possesses finite injective dimensions over itself on both the left and right sides.
We now turn our attention to a stronger variant of this Gorenstein property for noncommutative rings.

\begin{defn}\cite[Definition 1.3]{RWZ2021}\label{AS Gorenstein-defn}
	Let $R$ be a noetherian algebra.
	We say $R$ is {\it left AS Gorenstein} (where AS stands for Artin and Schelter) if
	
	(1) $_RR$ has finite injective dimension, say $d$, and
	
	(2) for every finite dimensional left $R$-module $V$, $\Ext^i_R(V, R) = 0$ for all $i \neq d$ and $\Ext^d_R(V, R)$ is also finite dimensional.
	
	Right AS Gorensteinness is analogously defined.
	We say $R$ is {\it AS Gorenstein} if it is both left and right AS Gorenstein.
	If further, $\gld R = d < + \infty$, then $R$ is called {\it AS regular}.
\end{defn}
If $R$ is an AS Gorenstein algebra of dimension $d$, then $\Ext^d_R(V, R) \neq 0$ for any finite dimensional $R$-module $V$. Refer to \cite[Lemma 1.4]{RWZ2021} for details.

%%As mentioned in the introduction, AS regular algebras are closely related to skew Calabi-Yau algebras.
%%%\cite[Definition 0.1]{RRZ2014}
%Let' recall the definition of skew Calabi-Yau algebras \cite[Definition 0.1]{RRZ2014} which are closely related to AS regular algebras.
%An algebra $R$ is called homological smooth if $R$ has a bounded resolution by finitely generated projective $R^e$-modules.
%Then $H$ is homological smooth if and only if $H$ is of type $\mathbf{FP}$, see \cite[Lemma 3.3]{LMeu2019} for example.

Let us revisit the definition of skew Calabi-Yau algebras as presented in \cite[Definition 0.1]{RRZ2014}, a class intimately related to AS regular algebras.
An algebra $R$ is called homologically smooth when it admits a bounded resolution consisting entirely of finitely generated projective modules over its enveloping algebra $R^e$.
According to \cite[Lemma 3.3]{LMeu2019}, for instance, a Hopf algebra $H$ is said to be homologically smooth if and only if it is of type $\mathbf{FP}$.

A homologically smooth algebra $R$ is called {\it skew Calabi-Yau} of dimension $d$ for some integer $d \geq 0$, if $R$ has an automorphism $\mu$ of $R$ such that
\[ \Ext^i_{R^e}(R, R^e) \cong \begin{cases}
	0, & i \neq d\\
	R_{\mu}, & i = d
\end{cases}\]
as $R$-$R$-bimodules, where $R_{\mu}$ signifies the $R$-$R$-bimodule twisted by $\mu$.
This bimodule is identical to $R$ as a $\kk$-vector space, with the action defined by $r \cdot x \cdot r' = rx\mu(r')$ for all $r, r' \in R$ and $x \in R_{\mu}$. % (equivalently, $x \in R$).

%If $R$ is AS Gorenstein of dimension $d$ then $\Ext^d_R(V, R)$ is simple right $R$-module for each finite dimensional simple left module $V$, and similarly on the other side, see \cite[Lemma 1.4]{RWZ2021}.
%This shows that Definition \ref{AS Gorenstein-defn} is equivalent to the definition of AS Gorenstein in \cite[Definition 3.1]{WZ2003} for any algebra $R$ for which all simple modules are finite-dimensional.
When an algebra $R$ is AS Gorenstein of dimension $d$, for each finite-dimensional simple left module $V$, the Ext group $\Ext^d_R(V, R)$ is also a simple right $R$-module, and a similar property holds on the right side, as demonstrated in \cite[Lemma 1.4]{RWZ2021}. This observation substantiates that Definition \ref{AS Gorenstein-defn} is indeed equivalent to the definition of AS Gorenstein provided in \cite[Definition 3.1]{WZ2003} for any algebra $R$ whose simple modules are all finite-dimensional.

%Compare with \cite[Definition 1.2]{BZ2008}, we do not require the Hopf algebra $H$ to be noetherian. %We substitute type $\mathbf{FP}_{\infty}$ for noetherian

In the following, we show that Definition \ref{AS Gorenstein-defn} is equivalent to the definition of AS Gorenstein in \cite[Definition 1.2]{BZ2008} for any noetherian Hopf algebra with a bijective antipode.

\begin{defn}\label{AS Gorenstein-Hopf-defn}
	Let $H$ be a Hopf algebra of type $\mathbf{FP}_{\infty}$ with a bijective antipode. 
	We say $H$ is {\it AS Gorenstein} if
	
	(1) $_HH$ has finite injective dimension, say $d$, and
	
	(2) $\Ext^i_H(\kk, H) = \begin{cases}
		0, & i \neq d \\
		\kk, & i = d
	\end{cases}$.
	
	If further, $\gld H = d$, then $H$ is called {\it AS regular}.
\end{defn}

Recall the standard actions by which tensor products and Hom-spaces of $H$-modules $U$ and $V$ become $H$-modules
\[ h \rhu (u \otimes v) = \sum_{(h)} (h_1 \rhu u) \otimes (h_2 \rhu v), \quad (h \rhu f)(u) = \sum_{(h)} h_1 \rhu f(Sh_2 \rhu u) \]
for all $h \in H$, $u \in U$, $v \in V$ and $f \in \Hom(U, V)$.
In case the $H$-action on $V$ is trivial, the second formula simplifies. In particular, for any $h \in H$, $u \in U$ and $f \in V^* = \Hom(V, \kk)$,
\[ (h \rhu f)(u) = \sum_{(h)} h_1 \rhu f(Sh_2 \rhu u) = \sum_{(h)} \vep(h_1) f(Sh_2 \rhu u) = f(Sh \rhu u). \]

Then the natural adjoint isomorphism
\[ \Hom(U \otimes V, W) \To \Hom(U, \Hom(V, W)), \]
restricts to an isomorphism
\[ \Hom_H(U \otimes V, W) \To \Hom_H(U, \Hom(V, W)), \]
see \cite{BG1997}.
Consequently, if $_HU$ is projective, then $U \otimes V$ is also a projective $H$-module by this adjoint isomorphism.

It is a standard fact that for any left $H$-modules $V$ and $W$, the natural map
\[ \varphi_{W, V} : W \otimes V^* \To \Hom(V, W), \qquad w \otimes f \mapsto f(v)w, \]
is an $H$-module map.
Further, $\varphi_{W, V}$ is bijective when $V$ is finite dimensional.
There is a morphism of right $H$-modules
\[ \psi_{U, W}: \Hom_{H}(U, H) \otimes W \To \Hom_{H}(U, H_H \otimes W), \qquad f \otimes w \mapsto \big( u \mapsto \sum_{(f(u))} f(u)_0 \otimes f(u)_1 \rhu w \big) \]
where the right $H$-module structure on $\Hom_{H}(U, H) \otimes W$ is given by 
\[ (f \otimes w) \lhu h = \sum_{(h)} (f \lhu h_1) \otimes (Sh_2 \rhu w)\]
for all $f \in \Hom_H(U, H)$, $w \in W$ and $h \in H$.
Further, $\psi_{U, W}$ is bijective when $U$ is finitely presented as an $H$-module.

The following lemma is well-known for instance \cite[Lemma 1.11 and its proof]{BG1997}.

\begin{lem}\label{AS Gorenstein-lem_0}
	Let $V$ be finite dimensional $H$-module, and $U$ be an $H$-module of type $\mathbf{FP_{\infty}}$.
	Then for all $i \geq 0$,
	\[ \Ext^i_H(U \otimes V, H) \cong \Ext^i_H(U, H) \otimes V^* \]
	as right $H$-modules, the right $H$-module structure on $\Ext^i_H(U, H) \otimes V^*$ is given by
	\[ (x \otimes f) \lhu h = \sum_{(h)} (x \lhu h_1) \otimes (Sh_2 \rhu f)\]
	for all $x \in \Ext^i_H(U, H)$, $f \in V^*$ and $h \in H$.
	%	\begin{enumerate}
		%		\item \cite[Proposition 1.3]{BG1997} $\Ext^i_H(U \otimes V, W) \cong \Ext^i_H(U,\Hom(V, W))$ for all $i$.
		%		\item If $V$ is finite dimensional and that $U$ is an $H$-module of type $\mathbf{FP_{\infty}}$, then 
		%		\[ \Ext^i_H(U \otimes V, H) \cong \Ext^i_H(U, H) \otimes V^* \]
		%		as right $H$-modules, the right $H$-module structure on $\Ext^i_H(U, H) \otimes V^*$ is given by $(x \otimes f) \lhu h = \sum_{(h)} (x \lhu h_1) \otimes (Sh_2 \rhu f)$ for all $x \in \Ext^i_H(U, H)$, $f \in V^*$ and $h \in H$.
		%	\end{enumerate}
\end{lem}

By using Ischebeck's spectral sequence, Brown and Zhang proved that the condition (2) of Definition \ref{AS Gorenstein-defn} can be replaced by a weaker form when $R = H$ is a Hopf algebra with a bijective antipode.

%\begin{lem}\cite[Lemma 3.2]{BZ2008}\label{AS Gorenstein-lem}
%	Let $H$ be a Hopf algebra of type $\mathbf{FP}_{\infty}$ with a bijective antipode. 
%	Suppose that $H$ has finite injective dimension $d$.
%	If the Ext groups $\Ext^i_{H}(\kk, H)$ are finite-dimensional over $\kk$ for all $i$, and if $\Ext^d_{H}(\kk, H) \neq 0$, then $H$ is AS Gorenstein.
%\end{lem}

\begin{lem}\cite[Lemma 3.2]{BZ2008}\label{AS Gorenstein-lem}
	Let $H$ be a Hopf algebra of type $\mathbf{FP}_{\infty}$ with a bijective antipode. 
	If $H$ has injective dimension $d$, then the following are equivalent.
	
	(1) For each finite dimensional $H$-module $V$, $\Ext^i_H(V, H) = 0$ for all $i \neq d$ and $\Ext^d_H(V, H)$ is finite dimensional.
	
	(2) $\Ext^i_H(\kk, H) = 0$ for all $i \neq d$ and $\Ext^d_H(\kk, H)$ is finite dimensional.
	
	(3) $\Ext^i_H(\kk, H) = \begin{cases}
		0, & i \neq d \\
		\kk, & i = d
	\end{cases}$.
\end{lem}
\begin{proof}
	(1) $\Rightarrow$ (2)
	It is trivial.
	
	(2) $\Rightarrow$ (3)
%	It follows from Lemma \ref{AS Gorenstein-lem}.%\cite[Lemma 3.2]{BZ2008}. 
	Given that $H$ is a Hopf algebra of type $\mathbf{FP}_{\infty}$ with finite injective dimension, Ischebeck's spectral sequence yields
	\[ \Ext_{H}^{p}(\Ext_{H^{op}}^{-q}(\kk, H), H) \Rightarrow \Tor_{-p-q}^H(\kk, H). \]
	From Lemma \ref{AS Gorenstein-lem_0}, we deduce that $\Ext_{H}^{d}(\kk, H) \otimes \Ext_{H^{op}}^{d}(\kk, H)^* \cong \Ext_{H}^{d}(\Ext_{H^{op}}^{d}(\kk, H), H) \cong \kk$.
	Thus, $\dim (\Ext_{H}^{d}(\kk, H)) = 1$, implying $\Ext_{H}^{d}(\kk, H) = \kk$.
	
	(3) $\Rightarrow$ (1)
	By Lemma \ref{AS Gorenstein-lem_0}, for any finite dimensional $H$-module $V$,
	\[\Ext^i_H(V, H) \cong \Ext^i_H(\kk, H) \otimes V^* \cong \begin{cases}
		0, & i \neq d \\
		V^*, & i = d.
	\end{cases}\]
	as right $H$-modules.
	This establishes the result.
\end{proof}

\subsection{Hopf module structure on $\Hom_{A^{op}}(M, X)$}

Let $H^*$ be the linear dual of $H$.
%An left $H^*$-module $V$ is called rational if $H^* \rhu v$ is finitely dimensional for any $v \in V$.
A left module $V$ over $H^*$ is said to be rational if for any $v \in V$, the submodule generated by the action of $H^*$ on $v$, denoted as $H^* \cdot v$, possesses finite dimension.

For any left $H^*$-module $V$, it is rational if and only if the $H^*$-action can be derived from a compatible right $H$-coaction.
In other words, that is, $V$ must be an $H$-comodule such that
$$\sum_{(v)} v_0 \langle f, v_1 \rangle = f \rhu v$$
for any $v \in V$ and $f \in H^*$.

Let $M$ be a vector space and $N$ be a right $H$-comodule.
%Now assume that the comodule structure on $R$ is trivial.
Then $\Hom(M, N)$ is a left $H^*$-module, under the action
$$(h^* \rhu f) (m) = h^* \rhu (f(m))$$
for all $h^* \in H^*$, $f \in \Hom(M, N)$ and $m \in M$.

%Let $A = B^{co H} \subseteq B$ be an $H$-Galois extension.
%Notice that $A$ can be seen as an $H$-comodule algebra with the trivial $H$-comodule structure.
%If $M \in {\M_A}$ and $N \in {\M^H_A}$, then $\Hom_{A^{op}}(M, N)$ is a $H^*$-submodule of $\Hom(M, N)$.
%In addition, if $M$ is finitely generated as an $R$-module, then $\Hom_{A^{op}}(M, N)$ is an $H^*$-submodule of a direct sum of some copies of $N$. Since $N$ is a rational $H^*$-module, it follows that $\Hom_{A^{op}}(M, N)$ is also rational. Hence $\Hom_{A^{op}}(M, N)$ is an $H$-comodule.
Then we have the following immediate consequence.

\begin{lem}\label{lem-0}
	Let $A = B^{co H} \subseteq B$ be an $H$-Galois extension, and $M$ be a right $A$-module and $X \in {\M^H_{A}}$.
	
	(1)
	If $M_A$ is finitely generated, then $\Hom_{A^{op}}(M, X)$ is an $H$-comodule.
	
	(2)
	If $M_A$ is finitely presented, then $\Hom_{A^{op}}(M, X)^{co H} = \Hom_{A^{op}}(M, X^{co H})$.
\end{lem}
\begin{proof}
	(1)
%	It is easy to check that $\Hom_{A^{op}}(M, X)$ is a $H^*$-submodule of $\Hom(M, X)$.
%	In addition, if $M$ is finitely generated as an $R$-module, then $\Hom_{A^{op}}(M, X)$ is an $H^*$-submodule of a direct sum of some copies of $X$.
%	Since $X$ is a rational $H^*$-module, it follows that $\Hom_{A^{op}}(M, X)$ is also rational.
	It is straightforward to verify that $\Hom_{A^{op}}(M, X)$ constitutes an $H^*$-submodule of $\Hom(M, X)$.
	Furthermore, when $M$ is finitely generated as an $A$-module, $\Hom_{A^{op}}(M, X)$ can be viewed as a subset of a direct sum of multiple copies of $X$. Given that $X$ is a rational $H^*$-module, it follows that $\Hom_{A^{op}}(M, X)$ inherits its rationality and thus becomes an $H$-comodule itself.
	
	Consequently, the $H$-coaction on $\Hom_{A^{op}}(M, X)$ is characterized by the following equation:
	\begin{equation}\label{H-comod-str-on-Hom-set}
		\sum_{(f)} f_0(m) \otimes f_1 = \sum_{(f(m))} f(m)_0 \otimes f(m)_1
	\end{equation}
	for all $f \in \Hom_{A^{op}}(M, X)$ and $m \in M$.
	
	(2)
	For any finitely generated $A$-module $N$, we have a canonical morphism
	\[ \alpha_N:\Hom_{A^{op}}(N, X^{co H}) \to \Hom_{A^{op}}(N, X)^{co H} \]
	which is functorial in both the variables $N$ and $X$.
	If further, $N$ is projective, then $\alpha_N$ is bijective.
	Since $M_A$ is finitely presented, there is a short exact sequence
	\[ P_1 \To P_0 \To M \To 0, \]
	where $P_0$ and $P_1$ are finitely generated projective $A$-module.
	Thus, we have a commutative diagram
	$$ \xymatrix{ 0 \ar[r] & \Hom_{A^{op}}(M, X^{co H}) \ar[r] \ar[d]^{\alpha_M} & \Hom_{A^{op}}(P_0, X^{co H}) \ar[r] \ar[d]^{\alpha_{P_0}} & \Hom_{A^{op}}(P_1, X^{co H}) \ar[d]^{\alpha_{P_1}} \\ 0 \ar[r] & \Hom_{A^{op}}(M, X)^{co H} \ar[r] & \Hom_{A^{op}}(P_0, X)^{co H} \ar[r] & \Hom_{A^{op}}(P_1, X)^{co H}}$$
	in which every row is exact.
	The map $\alpha_M$ is also bijective because $\alpha_{P_0}$ and $\alpha_{P_1}$ are both bijective.
\end{proof}

\begin{lem}\label{B-otimes-H-mod-str}
	Given a $B$-module $M$, finitely generated over $A$, and any Hopf module $X \in {\M^H_{B}}$, the Hom space $\Hom_{A^{op}}(M, X)$ naturally admits a Hopf module structure in $\M^H_{H}$.
\end{lem}
\begin{proof}
%	Recall that the right $H$-module of $\Hom_{A^{op}}(M, X) \cong \Hom_{A^{e}}(A, \Hom(M, X))$ is given by \eqref{H-mod-on-M^A}, that is,
%	\begin{equation}\label{H-mod-str-on-Hom-set}
%		(f \lhu h) (m) = \sum f(m\kappa^1(h)) \kappa^2(h).
%	\end{equation}
	Notice that the right $H$-module structure on $\Hom_{A^{op}}(M, X)$, equivalent to $\Hom(M, X)^A$, %$\Hom_{A^{e}}(A, \Hom(M, X))$,
	is expressed by \eqref{H-mod-on-M^A} as:
	\begin{equation}\label{H-mod-str-on-Hom-set}
		(f \lhu h) (m) = \sum f(m\kappa^1(h)) \kappa^2(h).
	\end{equation}
	
	For all $f \in \Hom_{A^{op}}(M, X)$, $m \in M$, $h \in H$ and $h^* \in H^*$,
	\begin{align*}
%		\sum_{(f \lhu h)} (f \lhu h)_0 (m) \otimes (f \lhu h)_1 \xlongequal{\eqref{H-comod-str-on-Hom-set}} & \sum_{( (f \lhu h)(m) )}(f \lhu h)(m)_0 \otimes (f \lhu h)(m)_1 \\
		\sum_{(f \lhu h)} (f \lhu h)_0 (m) \otimes (f \lhu h)_1 \xlongequal{\eqref{H-comod-str-on-Hom-set}} & \rho_X \big( (f \lhu h)(m) \big) \\
		\xlongequal{\eqref{H-mod-str-on-Hom-set}} & \rho_X \big( f(m\kappa^1(h)) \kappa^2(h) \big) \\
		= & \sum_{(f(m\kappa^1(h))), (\kappa^2(h))} f(m\kappa^1(h))_0\kappa^2(h)_0 \otimes f(m\kappa^1(h))_1\kappa^2(h)_1 \\
		\xlongequal[]{\eqref{kappa^1-kappa^2_0-kappa^2_1}} & \sum_{(h), (f(m\kappa^1(h_1)))} f(m\kappa^1(h_1))_0\kappa^2(h_1) \otimes f(m\kappa^1(h_1))_1 h_2 \\
		\xlongequal[]{\eqref{H-comod-str-on-Hom-set}} & \sum_{(h), (f)} f_0(m\kappa^1(h_1)) \kappa^2(h_1) \otimes f_1 h_2 \\
		\xlongequal[]{\eqref{H-mod-str-on-Hom-set}} & \sum_{(f), (h)} (f_0 \lhu h_1) (m) \otimes f_1h_2.
	\end{align*}
%	\begin{align*}
%		\big( h^* \rhu (f \lhu h) \big)(m) = & \Big( h^* \rhu \big(\sum \kappa^1(h) f \kappa^2(h)\big) \Big)(m) \\
%		= & h^* \rhu \Big( \big(\sum \kappa^1(h) f \kappa^2(h)\big)(m) \Big) \\
%		= & h^* \rhu \Big( \sum  f \big( m\kappa^1(h)\big)\kappa^2(h) \Big) \\
%		= & \sum f \big( m\kappa^1(h) \big)_0 \kappa^2(h)_0 \langle h^*, f\big( m\kappa^1(h)\big)_1 \kappa^2(h)_1  \rangle \\
%		\xlongequal[]{\eqref{kappa^1-kappa^2_0-kappa^2_1}} &  \sum_{(h)}  f \big( m \kappa^1(h_1) \big)_0 \kappa^2(h_1) \langle h^*, f \big( m\kappa^1(h) \big)_1 h_2 \rangle.
%	\end{align*}
%	It follows that $\Hom_{A}(M, X) \in \M^H_H$.
%	It follows that $\Hom_{A}(M, X)$ is a Hopf module in $\M^H_H$.
	Consequently, $\Hom_{A}(M, X)$ is endowed with a Hopf module structure in $\M^H_H$.
%	\begin{align*}
%		\rho(f \lhu h)(m) & \xlongequal[]{\eqref{H-mod-on-M^A}} \rho (\sum \kappa^1(h) f \kappa^2(h))(m) \\
%		= & \sum_{(m), (\kappa^2(h))} \kappa^1(h) m_0 \kappa^2(h)_0 \otimes m_1 \kappa^2(h)_1 \\
%		& \xlongequal[]{\eqref{kappa^1-kappa^2_0-kappa^2_1}} \sum_{(m), (h)} \kappa^1(h_1) m_0 \kappa^2(h_1) \otimes m_1 h_2 \\
%		= & \sum_{(m), (h)} (m_0 \lhu h_1) \otimes m_1h_2.
%	\end{align*}
%	Hence $\Hom_{A}(M, X) \in \M^H_H$.
	%	Clearly, for any $M \in {_{(B^e)^{u}}\!\M^H_B}$, $M^A$ is also a right $B$-module and the $B$-action satisfies the corresponding compatibility condition, that is, $M^A \in {\M^H_{B^u \otimes H}}$.
\end{proof}

\subsection{AS Gorenstein property of Hopf Galois extensions}

%\begin{lem}\label{(-)^A-free}
%	Let $V$ be a finite dimensional $B$-module.
%	If $A$ is noetherian, then
%	$$\Ext_{A}^i(V, B) \cong \Ext_{A^e}^i(A, \Hom(V, B))$$
%	is a free $H$-module for all $i \geq 0$.
%\end{lem}

\begin{lem}\label{Hom(V,-)-acyclic-object}
	Let $R$ be a ring and $V$ be a right $R$-module of type $\mathbf{FP_{\infty}}$.
	If $I_{\lambda}$ is an injective $R$-module for any $\lambda \in \Lambda$, then the direct sum $\bigoplus_{\lambda \in \Lambda} I_{\lambda}$ is acyclic with respect to the functor $\Hom_{R^{op}}(V, -)$, that is, $\Ext^i_{R^{op}}(V, \bigoplus_{\lambda \in \Lambda} I_{\lambda}) = 0$ for all $i > 0$.
\end{lem}
\begin{proof}
	By assumption, $V_R$ admits a projective resolution $P_{\bullet}$ where each $R$-module $P_i$ is finitely generated.
	Therefore, we have the following natural isomorphism
	\[ \Hom_{R^{op}}(P_i, \bigoplus_{\lambda \in \Lambda} I_{\lambda}) \cong \bigoplus_{\lambda \in \Lambda} \Hom_{R^{op}}(P_i,  I_{\lambda}).\]
	Given that for all $\lambda \in \Lambda$, the sequence
	\[ 0 \To \Hom_{R^{op}}(V, I_{\lambda}) \To \Hom_{R^{op}}(P_0, I_{\lambda}) \To \cdots \To \Hom_{R^{op}}(P_i, I_{\lambda}) \To \cdots \]
	is exact, it follows that the sequence
	\[ 0 \To \Hom_{R^{op}}(V, \bigoplus_{\lambda \in \Lambda} I_{\lambda}) \To \Hom_{R^{op}}(P_0, \bigoplus_{\lambda \in \Lambda} I_{\lambda}) \To \cdots \To \Hom_{R^{op}}(P_i, \bigoplus_{\lambda \in \Lambda} I_{\lambda}) \To \cdots \]
	also remains exact.
	Taking homologies then yields $\Ext^i_{R^{op}}(V, \bigoplus_{\lambda \in \Lambda} I_{\lambda}) = 0$ for all $i > 0$.
	This confirms that $\bigoplus_{\lambda \in \Lambda} I_{\lambda}$ is acyclic with respect to the functor $\Hom_{R^{op}}(V, -)$.
\end{proof}

\begin{lem}\label{(-)^A-free}
	Given a $B$-module $V$, if $V_A$ is of type $\mathbf{FP_{\infty}}$, then $\Ext_{A^{op}}^i(V, B) \cong \Ext_{A^{op}}^i(V, A) \otimes H$ as right $H$-modules for $i \geq 0$.
	%each $\Ext_{A^{op}}^i(V, B)$ is a free $H$-module of rank $\dim (\Ext_{A^{op}}^i(V, A))$ for $i \geq 0$.
\end{lem}
\begin{proof}
%	By Lemma \ref{Hopf-bimod-cat-is-Grothendieck}, we have an exact complex $0 \to B \to E^0 \to E^1 \to \cdots$ in $\M^H_B$ where $E^i \cong I^i \otimes H$ and $I^i$ is an injective $B$-module for any $i \in \N$.
	For any $X \in {\M^H_B}$, there is an injective right $B$-module $I$ with a $B$-module inclusion $\iota: X \hookrightarrow I$.
	Note that $I \otimes H$ is a Hopf module in ${\M^H_B}$ via
	\begin{equation*}%\label{I-otimes-H-Hopf-mod-str-eq}
		(x \otimes h) b = \sum_{(b)} xb_0 \otimes hb_1, \qquad \rho(x \otimes h) = \sum_{(h)} x \otimes h_1 \otimes h_2,
	\end{equation*}
	where $x \in I$, $h \in H$ and $b \in B$.
	There is a linear embedding $\widetilde{\iota}: X \hookrightarrow I \otimes H, \; x \mapsto \sum_{(x)} \iota(x_0) \otimes x_1$.
	For all $x \in X$ and $b \in B$, the following equalities hold:
	\[ \widetilde{\iota}(xb) = \sum_{(x), (b)} \iota(x_0b_0) \otimes x_1b_1 = \sum_{(x), (b)} \iota(x_0)b_0 \otimes x_1b_1, \qquad \text{ and} \]
	\[ \rho(\widetilde{\iota}(x)) = \rho(\sum_{(x)} \iota(x_0) \otimes x_1) = \sum_{(x)} \iota(x_0) \otimes x_1 \otimes x_2 = \sum_{(x)} \widetilde{\iota}(x_0) \otimes x_1.  \]
	These two conditions together imply that $\widetilde{\iota}: X \hookrightarrow I \otimes H$ defines a morphism in ${\M^H_B}$.
	By induction, we construct an exact complex $0 \to B \to E^0 \to E^1 \to \cdots$ in $\M^H_B$ where each term $E^i$ is isomorphic to $I^i \otimes H$ and $I^i$ is an injective $B$-module for any $i \in \N$.
	
	Since $_{A}B$ is flat, any injective right $B$-module $I$ is also injective as an $A$-module.
	Moreover, it's clear that $I \otimes H$ can be viewed as a direct sum of copies of $I$ as $A$-modules.
	By virtue of Lemma \ref{Hom(V,-)-acyclic-object}, $I \otimes H$ is $\Hom_{A^{op}}(V, -)$-acyclic.
	Thus, we infer that for all $i \geq 0$,
	\[ \Ext^i_{A^{op}}(V, B) \cong \mathrm{H}^i(\Hom_{A^{op}}(V, E^{\bullet})) \]
	as $H$-modules.
	
	As $\Hom_{A^{op}}(V, E^{\bullet})$ forms a complex in $\M^H_H$ by Lemma \ref{B-otimes-H-mod-str}, we obtain the following isomorphisms
	\begin{align*}
		\mathrm{H}^i(\Hom_{A^{op}}(V, E^{\bullet})) & \cong \mathrm{H}^i(\Hom_{A^{op}}(V, E^{\bullet}))^{co H} \otimes H & \\
		& \cong \mathrm{H}^i(\Hom_{A^{op}}(V, E^{\bullet})^{co H}) \otimes H & \text{ since $(-)^{co H}$ is exact on } \M^H_H \\
		& \cong \mathrm{H}^i(\Hom_{A^{op}}(V, (E^{\bullet})^{co H})) \otimes H & \text{ by Lemma \ref{lem-0} (2)}.
	\end{align*}
	
	Since the functor $(-)^{co H}: {\M^H_B} \to \M_A$ is an equivalence, the sequence
	\[ 0 \To A = B^{co H} \To (E^0)^{co H} \To (E^1)^{co H} \To \cdots\]
	is exact.
	As $- \otimes_A B: \M_A \to {\M^H_B}$ is the quasi-inverse of $(-)^{co H}$, we have $(I^i \otimes_A B)^{co H} \cong I^i$ as $A$-modules.
	Consider the composition of the following maps:
	\[ I^i \otimes_A B \stackrel{\cong}{\To} I^i \otimes_B (B \otimes_A B) \stackrel{\id_{I^i} \otimes \beta}{\To} I^i \otimes_B (B \otimes H) \stackrel{\cong}{\To} I^i \otimes H, \]
	where $x \otimes_A b \mapsto \sum_{(b)} xb_0 \otimes b_1$.
	This composition defines an isomorphism between $I^i \otimes_A B$ and $I^i \otimes H$ in $\M^H_B$.
	Thus, $(E^i)^{co H} = (I^i \otimes H)^{co H} \cong (I^i \otimes_A B)^{co H} \cong I^i$ as right $A$-modules.
	Therefore, the complex $(E^{\bullet})^{co H}$ constitutes an injective resolution of $A_A$.
%	Since
%	\[ \mathrm{H}^i(\Hom_{A^{op}}(V, (E^{\bullet})^{co H})) \cong \Ext^i_{A^{op}}(V, A) \]
%	for all $i \geq 0$, it follows that
	Then we derive the following isomorphism of right $H$-modules,
	\[ \Ext^i_{A^{op}}(V, B) \cong \mathrm{H}^i(\Hom_{A^{op}}(V, E^{\bullet})) \cong \mathrm{H}^i(\Hom_{A^{op}}(V, (E^{\bullet})^{co H})) \otimes H \cong \Ext^i_{A^{op}}(V, A) \otimes H. \]
	This leads us to our desired conclusion.
%	\[ \mathrm{H}^i(\Hom_{A^{op}}(V, (E^{\bullet})^{co H})) \cong \Ext^i_{A^{op}}(V, A) \otimes H\] in $\M^H_H$.
%	It follows that $\Ext^i_{A^{op}}(V, B) \cong \Ext^i_{A^{op}}(V, A) \otimes H$ as right $H$-modules.
%	So we get the conclusion.
\end{proof}
%\begin{proof}
%	%$(-)^A: {_{B^{\eta}}\M_B^H} \To {\M_H^H}$.
%	
%	According to \eqref{RHom_R^e<->RHom_R}, 
%	\[ \Ext_{A}^i(V, B) \cong \Ext_{A^e}^i(A, \Hom(V, B)). \]
%	Notice that $\Hom(V, B)$ can be seen as a Hopf module in ${_{B^{\eta}}\!\M_B^H}$ since $\Hom(V, B) \cong V^* \otimes B$.
%	By \cite[Lemma ]{}, $\Ext^i_{A^e}(A, M)$ is a Hopf module in $\M^H_H$ for any $M \in {_{B^{\eta}}\!\M_B^H}$ and $i \geq 0$.
%%	Since ${_{B^{\eta}}\!\M_B^H}$ has enough injective objects, there is an injective resolution
%%	\[ 0 \To \Hom(V, B) \To I^0 \To I^1 \To \cdots \To I^i \To \cdots \]
%%	of $\Hom(V, B)$ by Hopf modules in ${_{B^{\eta}}\!\M_B^H}$.
%	
%	Hence $\Ext_{A^e}^i(A, \Hom(V, B))$ is also a Hopf module in $\M^H_H$.
%	So all $\Ext_{A}^i(V, B)$ are free $H$-modules for $i \geq 0$.
%\end{proof}

%\begin{lem}\cite[Lemma 2.3 (5)]{WZ2003}
%	Let $A$ be a noetherian PI algebra, and $M$ be a noetherian $A$-module. If $\Tor^A_j(S, M) = 0$ for all $j \geq i$ and for all simple $A^{op}$-modules $S$, then $\pd M < i$.
%\end{lem}

Recall that an algebra $R$ is called affine if it is finitely generated as a $\kk$-algebra.
Let $\Kdim$ denote the Krull dimension.
A polynomial identity algebra, or PI algebra for short, is an algebra satisfying a polynomial identity. We refer to \cite[MR, Chapter 13]{MR2001} for some basic materials about PI algebras.

%Let's consider the following hypothesis.
Now, let's proceed with the following hypothesis.

\begin{hypo}\label{WZ-hypo}
	Let $R$ be a noetherian affine PI algebra. For any $i \geq 0$, %There is a subset $\Phi \subseteq \N$ so that the following conditions hold.
	\begin{enumerate}
		\item $\Ext^i_{R^{op}}(U, R) \neq 0$ if and only if $\Ext^i_{R^{op}}(V, R) \neq 0$ for all simple $R$-modules $U$ and $V$, and 
		\item $\Ext^i_{R^{op}}(-, R)$ is an exact functor on $R$-modules of finite length.
	\end{enumerate}
	%	\begin{enumerate}
		%		\item If $i \in \Phi$, then $\Ext^i_{R^{op}}(S, R) \neq 0$ for all simple right $A$-modules $S$, and $\Ext^i_{A^{op}}(-, A)$ is an exact functor on right $A$-modules of finite length.
		%		\item If $i \notin \Phi$, then $\Ext^i_{A^{op}}(S, A) = 0$ for all simple right $A$-modules $S$.
		%	\end{enumerate}
\end{hypo}

In \cite{WZ2003}, Wu and Zhang gave a useful criterion for noetherian affine PI algebras to be AS Gorenstein.

\begin{thm}\cite[Corollary 2.10, Proposition 3.2]{WZ2003}\label{Wu-Zhang-thm}
	If the Hypothesis \ref{WZ-hypo} holds, then
	\begin{enumerate}
		%		\item $\Phi$ is non-empty and consists of a single element $\{ d \}$;
		\item $\injdim({_RR}) = \injdim({R_R}) = d = \Kdim R < + \infty$;
		\item $R$ is AS Gorenstein of dimension $d$.
	\end{enumerate}
\end{thm}

Using the above theorem of Wu and Zhang, we promptly obtain the following theorem.

\begin{thm}\label{PI-Hopf-Galois-ext-is-AS-Gorenstein}
	Let $H$ be a Hopf algebra, and $B$ is a noetherian affine PI algebra which is a faithfully flat $H$-Galois extension of $A$.
	If $A$ is AS Gorenstein, then $B$ is also AS Gorenstein.
	Further, if $H$ is AS Gorenstein of dimension $d_H$, then $\injdim({B_B}) = \injdim({A_A}) + d_H$. %is AS Gorenstein of dimension $d_A + d_H$ where $d_A$ is the injective dimension of $A$.
	%	If $A$ and $H$ are both noetherian AS Gorenstein, then $B$ is also AS Gorenstein of dimension $d_A + d_H$.
\end{thm}
\begin{proof}
	Given that $B$ is an affine PI algebra, every simple $B$-module has finite dimension \cite[13.10.3(i)]{MR2001}.
	Let $U$ be a simple right $B$-module. Then $U$ is finite dimensional.
	Since $B$ is faithfully flat as a right $A$-module, $A$ is right noetherian.
	Consequently, $U$ is an $A$-module of type $\mathbf{FP_{\infty}}$.
	This yields the isomorphism $\Ext_{A^{op}}^i(U, B) \cong B \otimes_A \Ext_{A^{op}}^i(U, A)$ for all $i \geq 0$.
	If $A$ is AS Gorenstein of dimension $d_A$, then $\Ext_{A^{op}}^i(U, A) = 0$ for $i \neq d_A$ and $\Ext_{A^{op}}^{d_A}(U, A)$ is non-zero and finite dimensional.
	Hence $\Ext_{A^{op}}^i(U, B) \neq 0$ if and only if $i = d_A$.
	By Lemma \ref{H-Galois-spectral-sequenceII}, we infer that for any integer $i \geq 0$, 
	$$\Ext_{B^{op}}^i(U, B) \cong \begin{cases}
		0, & i < d_A; \\
		\Ext_{H^{op}}^{i - d_A}(\kk, \Ext_{A^{op}}^{d_A}(U, B)), & i \geq d_A. \\
	\end{cases}$$
	According to Lemma \ref{(-)^A-free}, $\Ext_{A^{op}}^{d_A}(U, B)$ constitutes a free $H$-module of rank $\dim(\Ext_{A^{op}}^{d_A}(U, A))$.
	Therefore,
	\[ \Ext^i_{B^{op}}(U, B) \neq 0 \Longleftrightarrow \Ext_{H^{op}}^{i - d_A}(\kk, \Ext_{A^{op}}^{d_A}(U, B)) \neq 0 \Longleftrightarrow \Ext_{H^{op}}^{i - d_A}(\kk, H) \neq 0.\]
	For another simple $B$-module $V$,
	\[ \Ext^i_{B^{op}}(U, B) \neq 0 \Longleftrightarrow \Ext_{H^{op}}^{i - d_A}(\kk, H) \neq 0 \Longleftrightarrow \Ext^i_{B^{op}}(V, B) \neq 0. \]
	Thus, $B$ fulfills Hypothesis \ref{WZ-hypo} (1).
	
	Consider an exact sequence $0 \to W' \to W \to W'' \to 0$ of finite length left $B$-modules.
	If $A$ is AS Gorenstein of dimension $d_A$, then we obtain a short exact sequence
	\begin{equation}\label{ex-seq-ext}
		0 \To \Ext_{A^{op}}^{d_A}(W'', B) \To \Ext_{A^{op}}^{d_A}(W, B) \To \Ext_{A^{op}}^{d_A}(W', B) \To 0
	\end{equation}
	in $\M^H_H$.
	By Lemma \ref{(-)^A-free}, each term in \eqref{ex-seq-ext} is a projective $H$-module, making \eqref{ex-seq-ext} split as an exact sequence of $H$-modules.
	Hence, the following sequence  %the short sequence
	\[ 0 \to \Ext_{H^{op}}^{i - d_A}(\kk, \Ext_{A^{op}}^{d_A}(W'', B)) \to \Ext_{H^{op}}^{i - d_A}(\kk, \Ext_{A^{op}}^{d_A}(W, B)) \to \Ext_{H^{op}}^{i - d_A}(\kk, \Ext_{A^{op}}^{d_A}(W', B)) \to 0 \]
	remains exact.
	Consequently, for all $i \geq 0$, we have the exact sequence %By Lemma \ref{H-Galois-spectral-sequenceII}, we have an exact sequence
	\[ 0 \To \Ext_{B^{op}}^{i}(W'', B) \To \Ext_{B^{op}}^{i}(W, B) \To \Ext_{B^{op}}^{i}(W', B) \To 0, \]
	showing that $B$ satisfies Hypothesis \ref{WZ-hypo} (2).
	Therefore, by Theorem \ref{Wu-Zhang-thm}, the desired conclusion is reached.
\end{proof}

%It is well known that AS Gorenstein noetherian affine PI algebras have various good homological properties, see \cite[Proposition 3.7]{WZ2003} for instance.
It is a well-established fact that AS Gorenstein noetherian affine PI algebras exhibit numerous favorable homological properties, see \cite[Proposition 3.7]{WZ2003} for instance.

\section{Homological dimension of the Ehresmann--Schauenburg bialgebroids}

%In this section, let's recall the definition and basic properties of Ehresmann--Schauenburg bialgebroids, and some homological properties of the Ehresmann--Schauenburg bialgebroids are indicated.

In this section, we shall revisit the definition and essential attributes of Ehresmann--Schauenburg bialgebroids, further elucidating certain homological properties inherent to these structures.

\subsection{Hopf bimodule categories and Ehresmann--Schauenburg bialgebroids}

Let $H$ be a Hopf algebra, $R$ and $T$ be two $H$-comodule algebras.
A relative Hopf bimodule $X \in {_R\!\M_T^H}$ is an object of both $_R\!\M^H$ and $\M_T^H$ with the same comodule structure and such that the two module structures make it an $R$-$T$-bimodule.

For any $R \otimes T^{op}$-module $M$, $M \otimes H$ can be viewed as a Hopf module in ${_R\!\M_T^H}$ through the structure defined by
\begin{equation}\label{Hopf-bimodule-str-on-M-otimes-H}
	r(m \otimes h)t = \sum_{(r), (t)} r_0 m t_0 \otimes r_1 h t_1, \qquad \rho(m \otimes h) = \sum_{(h)} m \otimes h_1 \otimes h_2,
\end{equation}
for any $m \in M$, $h \in H$, $r \in R$ and $t \in T$.
The functor $- \otimes H$ serves as a right adjoint to the forgetful functor from ${_R\!\M_T^H}$ to ${_R\!\M_T}$, that is, there exists a natural isomorphism
\begin{equation}\label{adjoint-for-Hopf-bimodule}
	\Hom_{_R\!\M_T}(X, M) \To \Hom_{_R\!\M^H_T}(X, M \otimes H), \;\; f \mapsto \big( x \mapsto \sum_{(x)} f(x_0) \otimes x_1 \big)
\end{equation}
for any $X \in {_R\!\M_T^H}$ and $M \in {_R\!\M_T}$.
The inverse map is given by $g \mapsto (\id_M \otimes \varepsilon) \circ g$.

%It is clear that the left (resp., right) adjoint functor of an exact functor preserves projective (resp., injective) objects.
%Since $- \otimes H: {_R\!\M_S} \to {_R\!\M_S^H}$ and the forgetful functor ${_R\!\M_S^H} \to {_R\!\M_S}$ are exact, we have following lemmas.
It is evident that the left (or right) adjoint functor of an exact functor preserves projective (or injective) objects. Given that$- \otimes H: {_R\!\M_T} \to {_R\!\M_T^H}$ and the forgetful functor ${_R\!\M_T^H} \to {_R\!\M_T}$ are both exact, the following lemmas can be derived.

\begin{lem}\label{proj-Hopf-bimodule-is-proj}
	(1) If $X$ is a projective object in ${_R\!\M_T^H}$, then $X$ is projective $R \otimes T^{op}$-module.
	
	(2) If $M$ is an injective $R \otimes T^{op}$-module, then $M \otimes H$ is an injective object in ${_R\!\M_T^H}$.
\end{lem}

%\begin{lem}\label{proj-Hopf-bimodule-is-proj}
%	If $X$ is a projective object in ${_R\!\M_S^H}$, then $X$ is projective $R \otimes S^{op}$-module.
%\end{lem}
%
%\begin{lem}\label{inj-bimodule-is-inj-Hopf-bimod}
%	If $M$ is an injective $R \otimes S^{op}$-module, then $M \otimes H$ is an injective object in ${_R\!\M_S^H}$.
%\end{lem}

\begin{lem}\label{Hopf-bimod-cat-is-Grothendieck}
	Let $H$ be a Hopf algebra, $R$ and $T$ be two $H$-comodule algebras.
	
	(1)
	Then ${_R\!\M^H_T}$ has enough injective objects.
	
	(2)
	For any $X \in {_R\!\M^H_T}$, there is an exact complex $0 \to X \to I^0 \otimes H \to I^1 \otimes H \to \cdots$ in $_R\!\M^H_T$ where $I^i$ is an injective $R \otimes T^{op}$-module for any $i \in \N$.
\end{lem}
\begin{proof}
	(1)
	For any $X \in {_R\!\M^H_T}$, there is an injective $R \otimes T^{op}$-module $I$ with a $R \otimes T^{op}$-module inclusion $\iota: X \hookrightarrow I$.
	By Lemma \ref{proj-Hopf-bimodule-is-proj}, $I \otimes H$ is an injective object.
	There is a linear embedding $\widetilde{\iota}: X \hookrightarrow I \otimes H, \; x \mapsto \sum_{(x)} \iota(x_0) \otimes x_1$.
	Since
	\[ \widetilde{\iota}(rxt) = \sum_{(r), (x), (t)} \iota(r_0 x_0 t_0) \otimes r_1x_1t_1 = \sum_{(r), (x), (t)} r_0 \iota(x_0) t_0 \otimes r_1x_1t_1, \qquad \text{ and} \]
	\[ \rho(\widetilde{\iota}(x)) = \rho(\sum_{(x)} \iota(x_0) \otimes x_1) = \sum_{(x)} \iota(x_0) \otimes x_1 \otimes x_2 = \sum_{(x)} \widetilde{\iota}(x_0) \otimes x_1,  \]
	it follows that $\widetilde{\iota}: X \hookrightarrow I \otimes H$ is a morphism in ${_R\!\M^H_T}$.
	This implies that ${_R\!\M^H_T}$ has enough injective objects.
	
	(2) can be proved by repeatedly applying (1).
%	
%	
%	By induction, we have an exact complex $0 \to X \to I^0 \otimes H \to I^1 \otimes H \to \cdots$ in $_R\!\M^H_S$ where $I^i$ is an injective $R \otimes S^{op}$-module for any $i \in \N$.
\end{proof}

For any $H$-comodule $V$, $R \otimes V \otimes T$ can be viewed as a Hopf module in ${_R\!\M_T^H}$ via
\begin{equation}\label{Hopf-bimodule-str-on-M-otimes-H'}
	r'(r \otimes v \otimes t)t' = r'r \otimes v \otimes tt', \qquad \rho(r \otimes v \otimes t) = \sum_{(r), (v), (t)} (r_0 \otimes v_0 \otimes t_0) \otimes r_1v_1t_1,
\end{equation}
for any $v \in V$, $r, r' \in R$ and $t, t' \in S$.
%One can easily check that $R \otimes - \otimes S: {\M^H} \to {_R\!\M_S^H}$ is a left adjoint of the forgetful functor from ${_R\!\M_S^H}$ to ${\M^H}$.
%Hence we have a natural isomorphism
%\begin{equation}\label{adjoint-for-Hopf-bimodule'}
%	\Hom_{\M^H}(R \otimes V \otimes S, X) \To \Hom_{_R\!\M^H_S}(V, X), \;\; f \mapsto \big( v \mapsto f(1_R \otimes v \otimes 1_B) \big)
%\end{equation}
%for any $X \in {_R\!\M_S^H}$ and $V \in {\M^H}$, and the inverse map is given by $g \mapsto \big( r \otimes v \otimes b \mapsto r g(v) b \big)$.
One can readily verify that the functor $R \otimes - \otimes T: {\M^H} \to {_R\!\M_T^H}$ serves as a left adjoint to the forgetful functor from ${_R\!\M_T^H}$ to ${\M^H}$.
As a result, there exists a natural isomorphism given by
\begin{equation}\label{adjoint-for-Hopf-bimodule'}
	\Hom_{\M^H}(V, X) \To \Hom_{_R\!\M^H_T}(R \otimes V \otimes T, X), \;\; f \mapsto \big( v \mapsto f(1_R \otimes v \otimes 1_T) \big)
\end{equation}
for any $X \in {_R\!\M_T^H}$ and $V \in {\M^H}$.
The inverse map is defined by $g \mapsto \big( r \otimes v \otimes t \mapsto r g(v) t \big)$.

Since $R \otimes - \otimes T: {_R\!\M_T} \to {_R\!\M_T^H}$ and the forgetful functor ${_R\!\M_T^H} \to {\M^H}$ are both exact, we have the following lemma.

\begin{lem}\label{inj-Hopf-bimodule-is-inj}
	(1) If $X$ is an injective object in ${_R\!\M_T^H}$, then $X$ is injective $H$-comodule.
	
	(2) If $V$ is a projective $H$-comodule, then $R \otimes V \otimes T$ is a projective object in ${_R\!\M_T^H}$.
\end{lem}

To any Hopf–Galois extension $A = B^{co H} \subseteq B$, one associates a $A$-coring $B \otimes H$ and a bialgebroid $L(B, B, H)$, see \cite{Sch1998} or \cite[section 34]{BW2003}.
%These can be viewed as a quantization of the gauge or Ehresmann groupoid that is associated to a principal fibre bundle, see \cite{Mac2005} for example.
%\subsection{$\times_A$-Bialgebras and a structure theorem of Hopf bimodule category}
%\subsection{A structure theorem of Hopf bimodule category}
%In this section we shall fix the (necessarily somewhat tedious) notations surrounding $\times_A$-bialgebras, following closely \cite{Tak1977}.
%The notion of a $\times_A$-bialgebra was first introduced by Sweedler \cite{Swe1974} for a commutative $A$ and then generalised by Takeuchi \cite{Tak1977} to an arbitrary $A$. In this section we briefly recall Takeuchi's definition (see \cite{Tak1977} for details).
%We refer the reader to \cite{Boe2009, Sch2000, Sch1998} for details about bialgebroids and $\times$-Hopf algebras.
%Throughout the section $A$ is a fixed $\kk$-algebra and $\overline{A} = A^{op}$. we denote by $A \to \overline{A}, \; a \mapsto \overline{a}$ the obvious $\kk$-linear isomorphism.
%
%Let $M$ be an $\overline{A}$-$\overline{A}$-bimodule, and $N$ be an $A$-$A$-bimodule. We use the following notations, due to MacLane, see \cite{Swe1974, Tak1977} (we shall only 'explain' them here by giving a couple of examples).
In the remainder of this section, $H$ is assumed to be a Hopf algebra with a bijective antipode, and $R$ and $B$ are considered as right $H$-comodule algebras.
Put
$$L(R, B, H) := (R \otimes B)^{co H} = R \Box_H B^{op}.$$
It is evident that $L(R, B, H)$ forms a subalgebra of $R \otimes B^{op}$.
We fix the formal notation
\[ l = \sum l^1 \otimes l^2 \in R \otimes B^{op} \]
for $l \in L(R, B, H)$, so that $ll' = \sum l^1{l'}^1 \otimes {l'}^2l^2$.

%For any Hopf module $M \in {_R\!\M^H_B}$, $M^{co H}$ is an $L$-module which is defined by
%\[ l \vtr m = \sum l^1 m l^2.\]
%
%(2) for $N \in {_L\!\M}$ we use $\vtr$ to denote the module structure, $N$ is a right $A$-module via $na = (1 \otimes a) \vtr n$, and the right $B$-module and $H$-comodule structures on $N \otimes_A B$ are those induced by the right tensorand, and the left $R$-module stucture is given by
%\begin{equation}\label{R-mod-str}
%	r(n \otimes_A b) = \sum \big( r_0 \otimes \kappa^1(r_1) \big) \vtr n \otimes_A \kappa^2(r_1) b.
%\end{equation}

In \cite{Sch1998}, Schauenburg extended Schneider's structure theorem for relative Hopf modules to a description of the structure of relative Hopf bimodules.

%\begin{lem}\label{coinvariant-tensor-module}
%	Let $H$ be a Hopf algebra, $B$ be a faithfully flat right $H$-Galois extension of $A$.
%	Let $S$ be a $\kk$-algebra and $M \in {_S\!\M^H_B}$, where $S$ has the trivial comodule structure.
%	Then for any $N \in \M_S$ we have $(N \otimes_S M)^{co H} \cong N \otimes_S M^{co H}$.
%\end{lem}

\begin{thm}\cite[Theorem 3.3]{Sch1998}\label{ff-Hopf-Galois-ext-Hopf-bimod-cat-str-thm}
	Let $H$ be a Hopf algebra and $B$ a right faithfully flat $H$-Galois extension of $A := B^{co H}$.
	Consider a right $H$-comodule algebra $R$, and denote $L:= L(R, B, H)$.
	There exists a category equivalence
%	\[ \mathcal{F}: {_R\!\M^H_B} \To {_L\!\M}, \;\; M \mapsto M^{co H}, \qquad \mathcal{G}: {_L\!\M} \To {_R\!\M^H_B}, \;\; N \mapsto N \otimes_A B. \]
	\[ F: {_R\!\M^H_B} \To {_L\!\M}, \;\; M \mapsto M^{co H}, \]
	with quasi-inverse
	\[ G: {_L\!\M} \To {_R\!\M^H_B}, \;\; N \mapsto N \otimes_A B. \]
	The $L$-module structure of $M^{co H}$ is defined by
	\[ l \vtr m = \sum l^1 m l^2;\]
	For $N \in {_L\!\M}$, the action $\vtr$ denotes the module structure, with $N$ being a right $A$-module via $na = (1 \otimes a) \vtr n$.
	The right $B$-module and $H$-comodule structures on $N \otimes_A B$ are induced by those of $B$, while the left $R$-module structure is given by
	\begin{equation}\label{R-mod-str}
		r(n \otimes_A b) = \sum_{(r)} \big( r_0 \otimes \kappa^1(r_1) \big) \vtr n \otimes_A \kappa^2(r_1) b.
	\end{equation}
\end{thm}

%Let $H$ be a Hopf algebra, $R$ be a right $H$-comodule algebra and $A \subseteq B$ be an $H$-Galois extension.
%Let $L := L(R, B, H)$.
%
%Notice that $L$ is a subalgebra of $R \otimes B^{op}$ such that $R \otimes B^{op}$ is a projective $L$-module.
%
%\[\xymatrix{{_R\!\M_B} \ar[r]^{- \otimes H} \ar[rd]^{} & {_R\!\M_B^H} \ar[d]^{(-)^{co H}} \\ & {_L\!\M} } \]

\subsection{Some homological properties of $L(R, B, H)$}

%In this subsection, we study the homological properties of the algebra $L(R, B, H)$ which is associated with an $H$-comodule algebra $R$ and a faithfully flat $H$-Galois extension $A=B^{co H} \subseteq B$.
In this subsection, we investigate the homological characteristics of the algebra $L := L(R, B, H)$ arising from an $H$-comodule algebra $R$ and a faithfully flat $H$-Galois extension $A=B^{co H} \subseteq B$.

\begin{lem}\label{L_A-proj}
	The algebra $L$ is projective as a right $A$-module.
\end{lem}
\begin{proof}
	Since $A \subseteq B$ is a faithfully flat $H$-Galois extension, there exists a right $H$-colinear and right $A$-linear splitting of the multiplication $B \otimes A \to B$ by \cite[Theorem 5.6]{SS2005}.
%	It follows that $B$ is a direct summand of some copies of $H$ as a $H$-comodule.
	Therefore, $L = (R \otimes B)^{co H}$ is a direct summand of $R \otimes A \otimes A_A = (R \otimes B \otimes A)^{co H}$ as a right $A$-module, implying that $L$ is a projective $A$-module.
\end{proof}

For any $M \in {_R\!\M_B}$, the left $L$-module structure of $M$ is obtained by restriction of scalars through the natural embedding map $L \hookrightarrow R \otimes B^{op}$.
As a result, $M \otimes_A B$ is identified as a Hopf bimodule within the category ${_R\!\M_B^H}$, as rigorously defined in Theorem \ref{ff-Hopf-Galois-ext-Hopf-bimod-cat-str-thm}.
%As a result, $M \otimes_A B$ is thereby identified as a Hopf module within the category ${_R\!\M_B^H}$, which is rigorously defined in Theorem \ref{ff-Hopf-Galois-ext-Hopf-bimod-cat-str-thm}.
%Hence $M \otimes_A B$ is a Hopf module in ${_R\!\M_B^H}$, which is defined in Theorem \ref{ff-Hopf-Galois-ext-Hopf-bimod-cat-str-thm}.

\begin{lem}\label{M-otimes_A-B-cong-M-otimes-H}
%	Let $M$ be a left $R \otimes B^{op}$-module. Then $M \otimes_A B \cong M \otimes H$ in ${_R\!\M_B^H}$.
	Let $M$ be a $R$-$B$-bimodule. Then $M \otimes_A B \cong M \otimes H$ in ${_R\!\M_B^H}$, where the Hopf module structure on $M \otimes H$ is given by \eqref{Hopf-bimodule-str-on-M-otimes-H}.
\end{lem}
\begin{proof}
	Utilizing the bijectivity of the Galois map $\beta: B \otimes_A B \to B \otimes H$, the composition map
	\[ \varphi: M \otimes_A B \stackrel{\cong}{\To} M \otimes_B B \otimes_A B \stackrel{\id_M \otimes \beta}{\To} M \otimes H, \quad m \otimes b \mapsto m \otimes 1_B \otimes b \mapsto \sum\limits_{(b)} mb_0 \otimes b_1 \]
	is also bijective.
	To establish that $\varphi$ is a morphism in ${_R\!\M_B^H}$, we recall that $M \otimes_A B$ is a left $R$-module via
	\begin{equation}\label{R-mod-str-on-M-otimes-B}
		r(m \otimes_A b) \xlongequal{\eqref{R-mod-str}} \sum_{(r)} \big( r_0 \otimes \kappa^1(r_1) \big) \vtr m \otimes_A \kappa^2(r_1) b = \sum_{(r)}  r_0m\kappa^1(r_1) \otimes_A \kappa^2(r_1) b. 
	\end{equation}
	For any $r \in R$, $m \in M$ and $b \in B$, we obtain that $\varphi\big( r(m \otimes_A b) \big)$ is equal to
	\begin{align*}
		 = & \varphi\big( \sum_{(r)} r_0m\kappa^1(r_1) \otimes_A \kappa^2(r_1) b \big) \\
		= & \sum_{(r), (\kappa^2(r_1)), (b)} r_0m\kappa^1(r_1)\kappa^2(r_1)_0 b_0 \otimes \kappa^2(r_1)_1 b_1 \\
		\xlongequal{\eqref{kappa^1-kappa^2_0-kappa^2_1}} & \sum_{(r), (b)} r_0m\kappa^1(r_1)\kappa^2(r_1) b_0 \otimes r_2 b_1 \\
		\xlongequal{\eqref{m-kappa=vep}} & \sum_{(r), (b)} r_0m b_0 \otimes r_1 b_1 \\
		\xlongequal{\eqref{Hopf-bimodule-str-on-M-otimes-H}} & \sum_{(r)} r\varphi(m \otimes b).
	\end{align*}
%	It follows that $\varphi$ is a left $R$-module morphism.
%	On the other hand, one can easily see that $\varphi$ is also a $B$-module and $H$-comodule morphism.
	This verifies that $\varphi$ is a left $R$-module morphism.
	Furthermore, it is straightforward to verify that $\varphi$ also preserves $B$-module and $H$-comodule structures.
	Hence $M \otimes_A B$ is isomorphic to $M \otimes H$ as Hopf modules in ${_R\!\M_B^H}$.
\end{proof}

\begin{lem}\label{Ext_L-Ext_R-otimes-B}
	For any left $L$-module $V$, $R \otimes B^{op}$-module $M$ and $i \in \N$, 
	\[ \Ext^i_{L}(V, M) \cong \Ext^i_{R \otimes B^{op}}(V \otimes_A B, M). \]
	In particular, if $M = R \otimes B$, then $\Ext^i_{L}(V, R \otimes B) \cong \Ext^i_{R \otimes B^{op}}(V \otimes_A B, R \otimes B)$ as $B$-$R$-bimodules.
\end{lem}
\begin{proof}
	Let $P_{\bullet}$ be a projective resolution of the left $L$-module $V$.
%	Each $P_{i} \otimes_A B$ is a projective object in ${_R\!\M_B^H}$ due to the equivalence of the functor $- \otimes_A B: {_L\!\M} \to {_R\!\M_B^H}$.
%	By Lemma \ref{proj-Hopf-bimodule-is-proj}, each $P_{i} \otimes_A B$ is a projective $R \otimes B^{op}$-module.
	Given that $- \otimes_A B: {_L\!\M} \to {_R\!\M_B^H}$ constitutes an equivalence of categories, it follows that each $P_{i} \otimes_A B$ is a projective object in ${_R\!\M_B^H}$ and also a projective $R \otimes B^{op}$-module, as established by Lemma \ref{proj-Hopf-bimodule-is-proj}.
	Due to the flatness of $_AB$, the sequence
	\[ \cdots \To P_1 \otimes_A B \To P_0 \otimes_A B \To V \otimes_A B \To 0 \]
	is exact, making $P_{\bullet} \otimes_A B$ a projective resolution of the left ${R \otimes B^{op}}$-module $V \otimes_A B$.
	By Lemma \ref{M-otimes_A-B-cong-M-otimes-H} and the adjoint isomorphism \eqref{adjoint-for-Hopf-bimodule}, we have the following isomorphisms,
%	\begin{align*}
%		\Hom_{L}(P_{\bullet}, M) \cong & \Hom_{_R\!\M_B^H}(P_{\bullet} \otimes_A B, M \otimes_A B) & \\
%		\cong & \Hom_{_R\!\M_B^H}(P_{\bullet} \otimes_A B, M \otimes H) & \text{ by Lemma \ref{M-otimes_A-B-cong-M-otimes-H} } \\
%		\cong & \Hom_{R \otimes B^{op}}(P_{\bullet} \otimes_A B, M) & \text{ by \eqref{adjoint-for-Hopf-bimodule},} 
%	\end{align*}
	\begin{align*}
		\Hom_{L}(P_{\bullet}, M) \cong & \Hom_{_R\!\M_B^H}(P_{\bullet} \otimes_A B, M \otimes_A B) \\
		\cong & \Hom_{_R\!\M_B^H}(P_{\bullet} \otimes_A B, M \otimes H) \\
		\cong & \Hom_{R \otimes B^{op}}(P_{\bullet} \otimes_A B, M).
	\end{align*}
	Upon taking homologies, we derive the isomorphism
	\[ \Ext^i_{L}(V, M) \cong \Ext^i_{R \otimes B^{op}}(V \otimes_A B, M)\]
	for all $i$.
%	The second assertion follows from the naturality of the above isomorphism.
	The subsequent claim is implied by the naturality of this isomorphism.
%	The second assertion is an immediate result of the naturality of the above isomorphism.
\end{proof}

\begin{lem}\label{proj-Hopf-bimod-lem}
	If $A \subseteq B$ is a faithfully flat $H$-Galois extension, then $R \otimes V \otimes B$ is a projective object in ${_R\!\M_B^H}$ for any $H$-comodule $V$.
\end{lem}
\begin{proof}
	Since $B$ is an injective $H$-comodule by Theorem \ref{faithful-flat-Hopf-Galois-extension-thm}, Lemma \ref{Hopf-module-injective} implies that all objects in ${\M_B^H}$ are injective $H$-comodules.
	For any Hopf module $X$ in ${_R\!\M_B^H}$, since $X$ is also an object in ${\M_B^H}$, $X$ is injective as an $H$-comodule.
	It follows that any short exact sequence
	\[ 0 \To X' \To X \To X'' \To 0 \]
	in ${_R\!\M_B^H}$ splits as a sequence of $H$-comodules.
	Therefore, the induced sequence
	\[ 0 \To \Hom_{\M^H}(V, X') \To \Hom_{\M^H}(V, X) \To \Hom_{\M^H}(V, X'') \To 0 \]
	is also exact.
	It follows from \eqref{adjoint-for-Hopf-bimodule'} that $R \otimes V \otimes B$ is a projective object in ${_R\!\M_B^H}$.
\end{proof}

\begin{prop}\label{R-otimes-B-is-L-proj-generator}
	Keep the notation as above. Then the following statements hold. \\
	(1) $R \otimes B$ is a projective $L$-module. \\
	(2) $L$ is a direct summand of $R \otimes B$ as an $L$-module.
%$R \otimes B$ is a projective generator as a $L$-module.
\end{prop}
\begin{proof}
	Given that the pair $\big( - \otimes_A B, (-)^{co H} \big)$ establishes an equivalence between ${_L\!\M}$ and ${_R\!\M_B^H}$ by Theorem \ref{ff-Hopf-Galois-ext-Hopf-bimod-cat-str-thm}, it suffices to prove the following: \\
	(1) $R \otimes B \otimes_A B$ is a projective object in ${_R\!\M_B^H}$, and \\
	(2) $L \otimes_A B$ is a direct summand of $R \otimes B \otimes_A B$ within the category ${_R\!\M_B^H}$.
	
	(1)
	Observe that we have the following isomorphism in ${_R\!\M_B^H}$,
	\[ R \otimes H \otimes B \stackrel{\cong}{\To} R \otimes B \otimes H, \quad r \otimes h \otimes b \mapsto \sum_{(r), (b)} r_0 \otimes b_0 \otimes r_1hb_1, \]
	with the inverse
	\[ R \otimes B \otimes H \To R \otimes H \otimes B, \qquad r \otimes b \otimes h \mapsto \sum_{(r), (b)} r_0 \otimes b_0 \otimes (Sr_1)h(S^{-1}b_1). \]
	The Hopf module structures of $R \otimes H \otimes B$ and $R \otimes B \otimes H$ are defined by \eqref{Hopf-bimodule-str-on-M-otimes-H'} and \eqref{Hopf-bimodule-str-on-M-otimes-H}, respectively.
	By Lemma \ref{M-otimes_A-B-cong-M-otimes-H}, we have
	\[ R \otimes B \otimes H \cong R \otimes B \otimes_A B\]
	in ${_R\!\M_B^H}$.
	Since $R \otimes H \otimes B$ is a projective object in ${_R\!\M_B^H}$ by Lemma \ref{proj-Hopf-bimod-lem}, it follows that $R \otimes B \otimes_A B$ is also a projective object.
	
	(2)
	Consider the exact sequence of $H$-comodules,
	\[ 0 \To \kk \To H \To \overline{H}:= H/\kk 1_H \To 0. \]
%	\[ 0 \To \kk 1_H \To H \To \overline{H}:= H/\kk 1_H \To 0. \]
	Applying the functor $R \otimes - \otimes B$ yields an exact sequence
	\[ 0 \To R \otimes \kk \otimes B \To R \otimes H \otimes B \To R \otimes \overline{H} \otimes B \To 0. \]
%	\[ 0 \To R \otimes \kk 1_H \otimes B \To R \otimes H \otimes B \To R \otimes \overline{H} \otimes B \To 0. \]
	in ${_R\!\M_B^H}$.
	By Lemma \ref{proj-Hopf-bimod-lem}, this sequence splits, implying that $R \otimes B (\cong R \otimes \kk \otimes B)$ is a direct summand of $R \otimes B \otimes_A B (\cong R \otimes H \otimes B)$ within ${_R\!\M_B^H}$.
	
	Thus, the proof is concluded.
	%This completes the proof.
\end{proof}

\begin{lem}\label{Ext-L-R-otimes-B^op-nonzero}
%	Let $H$ be a Hopf algebra, $R$ be a right $H$-comodule algebra and $A \subseteq B$ be a faithfully flat $H$-Galois extension.
	Let $V$ be a left $L$-module.
	\begin{enumerate}
		\item If $\Ext^i_{L}(V, L) \neq 0$, then $\Ext^i_{R \otimes B^{op}}(V \otimes_A B, R \otimes B) \neq 0$.
		\item Suppose that $V$ is an $L$-module of type $\mathbf{FP_{\infty}}$. If $\Ext^i_{R \otimes B^{op}}(V \otimes_A B, R \otimes B) \neq 0$, then $\Ext^i_{L}(V, L) \neq 0$.
	\end{enumerate}
\end{lem}
\begin{proof}
	(1)
	If $\Ext^i_{L}(V, L) \neq 0$, then $\Ext^i_{L}(V, R \otimes B) \neq 0$ as $L$ is a direct summand of $R \otimes B^{op}$, per Lemma \ref{R-otimes-B-is-L-proj-generator} (2).
	By Lemma \ref{Ext_L-Ext_R-otimes-B}, we obtain the isomorphism,
	\[ \Ext^i_{R \otimes B^{op}}(V \otimes_A B, R \otimes B) \cong \Ext^i_{L}(V, R \otimes B). \]
	Consequently, $\Ext^i_{R \otimes B^{op}}(V \otimes_A B, R \otimes B) \neq 0$ as well.
	
	(2)
	As $R \otimes B$ is a projective $L$-module by Lemma \ref{R-otimes-B-is-L-proj-generator} (1), there exists a free $L$-module $F$ with $\Ext^i_{L}(V, F) \neq 0$.
	If $V$ is an $L$-module of type $\mathbf{FP_{\infty}}$, then $\Ext^i_{L}(V, F)$ is isomorphic to a non-zero direct sum of some copies of $\Ext^i_{L}(V, L)$.
	Therefore, $\Ext^i_{L}(V, L)$ is also non-zero.
\end{proof}

%The following result is well known, see \cite[Proposition VIII. 6.7]{Bro1982} for example.
The following result is a well-established fact, as demonstrated in \cite[Proposition VIII. 6.7]{Bro1982} for instance.

\begin{lem}\label{proj-dim-pseudo-coherent-mod}
	%	Let $R$ be a ring.
	If $M$ is an $R$-module of type $\mathbf{FP}$, then $\pd_R(M) = \sup \{ i \mid \Ext^i_R(M, R) \neq 0 \}$.
\end{lem}
\begin{proof}
	Let $\pd_R(M) = n < + \infty$.
	Consequently, there exists an $R$-module $N$ such that $\Ext^n_R(M, N) \neq 0$.
	Now, let $F$ be a free $R$-module with a surjective $R$-module morphism $F \twoheadrightarrow N$.
	Since $\pd_R(M) = n$, the induced morphism $\Ext^n_R(M, F) \to \Ext^n_R(M, N)$ is surjective.
	Hence $\Ext^n_R(M, F) \neq 0$.
	Since $M$ is a  $R$-module of type $\mathbf{FP}$, we have an isomorphism $\Ext^n_R(M, R) \otimes_R F \cong \Ext^n_R(M, F)$.
	%Since $A$ is a noetherian ring and $B$ is a finitely generated $A$-module, 
	It follows that $\Ext^n_R(M, R) \otimes_R F$ is also non-zero.
\end{proof}

\begin{lem}\label{pd-V-otimes-B-leq-V}
	For any left $L$-module $V$, $\pd_{R \otimes B^{op}}(V \otimes_A B) \leq \pd_{L}(V)$.
	If $V$ is an $L$-module of type $\mathbf{FP}$, then $\pd_{R \otimes B^{op}}( V \otimes_A B) = \pd_L(V)$.
\end{lem}
\begin{proof}
	Assume that $\pd_L(V) = d < \infty$.
	Let $P_{\bullet}$ be a projective resolution of $V$ of length $d$.
	Given that the functor $- \otimes_A B: {_L\!\M} \to {_R\!\M_B^H}$ is an equivalence, $P_{\bullet} \otimes_A B$ constitutes a projective resolution of $V \otimes_A B$ in ${_R\!\M_B^H}$.
	Furthermore, each $P_{i} \otimes_A B$ is a projective $R \otimes B^{op}$-module due to Lemma \ref{proj-Hopf-bimodule-is-proj} (1).
	Consequently, the projective dimension of $V \otimes_A B$ over $R \otimes B^{op}$ is at most $d$, i.e., $\pd_{R \otimes B^{op}}(V \otimes_A B) \leq d$.
	
	Since $\Ext^d_{L}(V, L) \neq 0$ by Lemma \ref{proj-dim-pseudo-coherent-mod}, the non-vanishing of $\Ext^d_{R \otimes B^{op}}(V \otimes_A B, R \otimes B^{op})$ follows from Lemma \ref{Ext-L-R-otimes-B^op-nonzero}.
	Therefore, the projective dimension of $V \otimes B$ as an $R \otimes B^{op}$-module is exactly $d$.
\end{proof}

The following proposition is an immediate consequence of Lemma \ref{pd-V-otimes-B-leq-V}

\begin{prop}
	If $\lgld(L) < \infty$, then $\lgld(L) \leq \lgld(R \otimes B^{op})$.
\end{prop}

\subsection{Injective dimension of the Ehresmann--Schauenburg bialgebroids}

Indeed, ${_B\!\M^H_B}$ is a monoidal category equipped with the tensor product over $B$.
While $L:=L(B,B,H)$ is not generally a bialgebra, ${_L\!\M}$ is nonetheless a monoidal category, its structure being induced by ${_B\!\M^H_B}$.
In fact, $L$ is a bialgebroid, as defined in \cite{Sch1998}.
This notion of a bialgebroid extends the concept of a bialgebra to a non-commutative base ring context.
$L(B,B,H)$ is recognized as the Ehresmann--Schauenburg bialgebroid, as cited in \cite[34.14]{BW2003}.

\begin{lem}\label{pdim-R^e-R}
	Let $M$ be a $B^e$-bimodule.
	If $M_B$ or $_BM$ is projective, then $\pd_{B^e}(M) \leq \pd_{B^e}(B)$.
\end{lem}
\begin{proof}
	Assuming $\pd_{B^e}(B) = d < + \infty$, consider a projective resolution $P_{\bullet}$ of $_{B^e}B$ with length $d$.
	If $M_B$ is projective, then $B^e \otimes_B M$ or equivalently $B \otimes M$ becomes a projective $B^e$-module.
	Since $P_{\bullet}$ splits as a complex of right $B$-modules, $P_{\bullet} \otimes_B M$ serves as a projective resolution of $_{B^e}M$.
	Consequently, $\pd_{B^e}(M) \leq d = \pd_{B^e}(B)$.
\end{proof}

%\begin{lem}\label{gldim-tensor-cat-pdim-unit}
%	Let $(\A, \otimes, I)$ be a monoidal category that is at the same time abelian.
%	Suppose that $\A$ has enough projectives and that ? is exact for ? projective.
%\end{lem}

\begin{lem}\label{gldim-tensor-cat-pdim-unit}
%	Let $H$ be a Hopf algebra and $A \subseteq B$ be a faithfully flat $H$-Galois extension.
%	Let $L = L(B, B, H)$.
	If $V$ is a left $L$-module, then $\pd_{B^e}(V \otimes_A B) \leq \pd_{B^e}(B) + \rgld(A)$.
\end{lem}
\begin{proof}
	Without loss of generality, assume that $\rgld(A) = d < + \infty$.
	Let $V$ be a left $L$-module.
	Then there is an exact sequence of $L$-modules
	\[ 0 \To P_d \To P_{d-1} \To \cdots \To P_0 \To V \To 0, \]
	where each $L$-module $P_i$ is projective for all $i < d$.
	Since $L_A$ is projective by Lemma \ref{L_A-proj} and that $\rgld(A) = d$, each $P_i$ is projective as an $A$-module for all $i \leq d$.
	Consequently, $P_i \otimes_A B$ is a projective $B$-module for all $i \leq d$.
	By Lemma \ref{pdim-R^e-R}, $\pd_{B^e}(P_i \otimes_A B) \leq \pd_{B^e}(B)$ for all $i \leq d$.
	Applying the functor $- \otimes_A B$, we obtain the following exact sequence of $B^e$-modules
	\[ 0 \To P_d \otimes_A B \To P_{d-1} \otimes_A B \To \cdots \To P_0 \otimes_A B \To V \otimes_A B \To 0. \]
	Therefore, $\pd_{B^e}(V \otimes_A B) \leq \pd_{B^e}(B) + d = \pd_{B^e}(B) + \rgld(A)$.
\end{proof}

\begin{prop}\label{id(L)-leq-pd(B)}
	Then $\injdim(_LL) \leq \pd_{B^e}(B) + \rgld(A) \leq \gld(H) + \pd_{A^e}(A) + \rgld(A)$.
\end{prop}
\begin{proof}
	By Lemma \ref{Ext-L-R-otimes-B^op-nonzero} (1) and \ref{gldim-tensor-cat-pdim-unit}, we have
	\begin{align*}
		\injdim(_LL) = & \sup \{ i \mid \Ext^i_L(V, L) \neq 0 \text{ for some left $L$-module $V$} \} \\
		\leq & \sup \{ i \mid \Ext^i_{B^e}(V \otimes_A B, B^e) \neq 0 \text{ for some left $L$-module $V$} \} \\
		\leq & \sup \{ \pd_{B^e}(V \otimes_A B) \mid \text{for some left $L$-module $V$} \} \\
		\leq & \pd_{B^e}(B) + \rgld(A).
	\end{align*}
	The second inequality follows from Corollary \ref{pdim-B-as-B-B-bimodule}.
\end{proof}

\section{Homological properties under monoidal Morita--Takeuchi equivalences}

%In this section, we specialize further to the case that $A = \kk$.
In this section, we focus specifically on the case where $A = \kk$.
%In this case we have seen in \cite{Sch1996}, $L:= L(B, B, H)$ is a Hopf algebra with comultiplication
As established in \cite{Sch1996}, under this condition, $L:= L(B, B, H)$ becomes a Hopf algebra.
The comultiplication of $L$ is given by
\[ \Delta(l) = \sum_{(l^1)} {l^1}_0 \otimes \kappa^1({l^1}_1) \otimes \kappa^2({l^1}_1) \otimes l^2,\]
where $l = \sum l^1 \otimes l^2 \in L \subseteq B^e$.
The counit of $L$ is defined as $\vep(l) = \sum l^1l^2$, and the antipode is given by $S(l) = \sum {l^2}_0 \otimes \kappa^1({l^2}_1) l^1 \kappa^2({l^2}_1)$.

%Moreover, we have seen in \cite{Sch1996} that $B$ is a left $L$-comodule algebra with comodule structure $\delta: B \to L \otimes B$ via
%\[ b \mapsto \sum_{(b)} \big( b_0 \otimes \kappa^1(b_1) \big) \otimes \kappa^2(b_1). \]
%Then $B$ is an $L$-$H$-biGalois extension of $\kk$, that is, $B$ is both a right $H$-Galois extension of $\kk$ and a left $L$-Galois of $\kk$ such that the two comodule structures make it an $L$-$H$-bicomodule.

Furthermore, as detailed in \cite{Sch1996}, $B$ is endowed with a left $L$-comodule algebra structure via the comodule map $\delta: B \to L \otimes B$ defined as
\begin{equation}\label{L-comod-on-B}
	b \mapsto \sum_{(b)} \big( b_0 \otimes \kappa^1(b_1) \big) \otimes \kappa^2(b_1).
\end{equation}
This structure confers upon $B$ the property of being an $L$-$H$-biGalois extension of $\kk$.
In other words, $B$ is both a right $H$-Galois extension of $\kk$ and a left $L$-Galois of $\kk$ such that the two comodule structures make it an $L$-$H$-bicomodule.

\subsection{Monoidal Morita--Takeuchi equivalences preserve the type $\mathbf{FP}_{\infty}$.}

%Although there is an monoidal Morita--Takeuchi equivalence which doesn't preserve the global dimension, the type $\mathbf{FP}_{\infty}$ of Hopf algebras is an invariant of monoidal Morita--Takeuchi equivalences.

In the following, we demonstrate that the type $\mathbf{FP}_{\infty}$ of Hopf algebras is an invariant of monoidal Morita--Takeuchi equivalences.

%For any $H$-comodule $V$, $B \otimes V \otimes B$ is a Hopf module in ${_B \mathcal{M}_B^H}$. The $B$-$B$-bimodule structure is induced by the first and third tensorands, and the $H$-comodule structure is induced by the diagonal coaction.
%One can easily check that $B \otimes - \otimes B: {\M^H} \to {_B\!\M_B^H}$ is a left adjoint of the forgetful functor.
%So we have a natural isomorphism
%\begin{equation}\label{B-B-^H-adj}
%	\Hom_{\M^H}(V, M) \To \Hom_{_B\!\M_B^H}(B \otimes V \otimes B, M), \;\; f \mapsto \big( b \otimes v \otimes b' \mapsto bf(v)b' \big)
%\end{equation}
%for any $V \in {\M^H}$ and $M \in {_B\!\M_B^H}$, and the inverse map is given by $g \mapsto \big(v \mapsto g(1_B \otimes v \otimes 1_B) \big)$.

Recall that for any $H$-comodule $V$, by \eqref{Hopf-bimodule-str-on-M-otimes-H'}, $B \otimes V \otimes B$ naturally acquires a structure of a Hopf bimodule in the category ${_B \mathcal{M}_B^H}$.
The $B$-$B$-bimodule structure is induced by the first and third tensorands, and the $H$-comodule structure arises from the diagonal coaction.
Consequently, the functor $B \otimes - \otimes B$ serves as a left adjoint of the forgetful functor from ${\M^H}$ to ${_B\!\M_B^H}$, established by the existence of the following natural isomorphism
\begin{equation}\label{B-B-^H-adj}
	\Hom_{\M^H}(V, M) \To \Hom_{_B\!\M_B^H}(B \otimes V \otimes B, M), \;\; f \mapsto \big( b \otimes v \otimes b' \mapsto bf(v)b' \big)
\end{equation}
for any $V \in {\M^H}$ and $M \in {_B\!\M_B^H}$.

%Then we have the following lemma.
%\begin{lem}\label{proj-module-is-proj-as-Hopf-bimodule}
%	(1) If $V$ is projective $B^e$-module, then $B \otimes V \otimes B$ is a projective object in ${_B\!\M_B^H}$.
%	
%	(2) If $M$ is an injective object in ${_B\!\M_B^H}$, then $M$ is an injective $B^e$-module.
%\end{lem}

\begin{lem}\label{type-FP-Hopf-modules}
	Let $M$ be a Hopf module in ${_B\mathcal{M}_B^H}$.
	
	(1) Then $M$ is finitely generated as a $B^e$-module if and only if there is a finite-dimensional $H$-comodule $V$ with a surjective morphism $\pi: B \otimes V \otimes B \to M$ in $_B\mathcal{M}_B^H$.
	
	(2) If $M$ is a $B^e$-module of type $\mathbf{FP}_n$, then there exists an exact sequence
	\[ \cdots \To B \otimes V_i \otimes B \To \cdots \To B \otimes V_0 \otimes B \To M \To 0 \;\; \text{ in } {_B\mathcal{M}_B^H}, \]
	where each $V_i$ is an $H$-comodules for $i \geq 0$ and $V_i$ is finite-dimensional for $i \leq n$.
\end{lem}
\begin{proof}
	(1) Let $\{x_1,\dots,x_m\}$ be a set generators of $_{B^e}M$. 
	The Finiteness Theorem \cite[5.1.1]{Mon1993} guarantees the existence of a finite-dimensional $H$-subcomodule $V$ of $M$ that includes $\{x_1, \dots, x_m\}$.
	Consequently, there is a surjective morphism $\pi: B \otimes V \otimes B \twoheadrightarrow X, \; b \otimes v \otimes b' \mapsto bvb'$ in $_B\mathcal{M}^H_B$.
	
	(2) From (1), there exists a finite-dimensional  left $H$-comodule $V_0$ with a surjective morphism $\pi : B \otimes V_0 \otimes B \twoheadrightarrow M$ in $_B\mathcal{M}_B^H$.
	Schanuel’s lemma implies that the kernel of $\pi$ is a $B^e$-module of type $\mathbf{FP}_{n-1}$.
	The desired conclusion then follows by induction.
\end{proof}

%The following result is well known, and we give a short proof.
The subsequent lemma is a well-established fact, and we present a concise proof for it.
\begin{lem}\label{fg-equi-lem}
	Let $M$ be a module over a ring $R$. If $\underrightarrow{\lim} \Hom_R(M, N_i) = 0$ for any direct system $\{ N_i \}$ with $\underrightarrow{\lim} N_i=0$, then $M$ is finitely generated.
\end{lem}
\begin{proof}
	Consider the direct system $\{ M/M' \}$ indexed by the finitely generated submodules $M'$ of $M$.
	We have $\underrightarrow{\lim} M/M' = 0$, which implies $\underrightarrow{\lim}\Hom_R(M, M/M') = 0$.
	Consequently, there exists a finitely generated submodule $M'$ for which the induced map $\overline{\id_M}$ is the zero map in $\Hom_R(M, M/M')$, where $\id_M$ denotes the identity map on $M$.
	Therefore, we conclude that $M = M'$, establishing that $M$ is a finitely generated $R$-module.
\end{proof}

\begin{lem}\label{B-otimes-V-otimes-B-fg-proj}
	If $V$ is a finite dimensional $H$-comodule, then $(B \otimes V \otimes B)^{co H}$ is a finitely generated projective $L$-module.
\end{lem}
\begin{proof}
%	Since ${\M_B^H} \cong \M_{\kk}$, any short exact sequence $0 \to M' \to M \to M'' \to 0$ in ${_B\!\M_B^H}$ is split exact in $\M^H$.
%	% each Hopf module $M$ in ${\M_B^H}$ is injective as an $H$-comodule.
%	Hence
%	\[ 0 \To \Hom_{\M^H}(V, M') \To \Hom_{\M^H}(V, M) \To \Hom_{\M^H}(V, M'') \To 0 \]
%	is exact.
	
	Recall that the functor $(-)^{co H}: {_B\!\M_B^H} \to {_L\!\M}$ establishes an equivalence.
	Given that $B \otimes V \otimes B$ is a projective object in ${_B\!\M_B^H}$ by Lemma \ref{proj-Hopf-bimod-lem}, $(B \otimes V \otimes B)^{co H}$ is a projective $L$-module.
	
%	By Lemma \ref{fg-equi-lem}, to prove that $(B \otimes V \otimes B)^{co H}$ is finitely generated, we only need to prove that $\Hom_{L}((B \otimes V \otimes B)^{co H}, -)$ commutes with direct limit, that is, for any direct system $\{ N_i \}$ of $L$-modules, there is a natural isomorphism
	To demonstrate that $(B \otimes V \otimes B)^{co H}$ is finitely generated, we utilize Lemma \ref{fg-equi-lem}, which asserts that it suffices to show that the functor $\Hom_{L}((B \otimes V \otimes B)^{co H}, -)$ commutes with direct limits. In other words, for any direct system $\{ N_i \}$ of $L$-modules, there must exist a natural isomorphism
	\[ \Hom_{L}((B \otimes V \otimes B)^{co H}, \underrightarrow{\lim} N_i) \cong \underrightarrow{\lim} \Hom_{L}((B \otimes V \otimes B)^{co H}, N_i). \]
	By Theorem \ref{ff-Hopf-Galois-ext-Hopf-bimod-cat-str-thm}, we establish the following isomorphisms
	\begin{align*}
		\Hom_{L}((B \otimes V \otimes B)^{co H}, \underrightarrow{\lim} N_i) \cong & \Hom_{_B\!\M_B^H}((B \otimes V \otimes B), (\underrightarrow{\lim} N_i) \otimes B) \\
		\cong & \Hom_{_B\!\M_B^H}(B \otimes V \otimes B, \underrightarrow{\lim} (N_i\otimes B)) \\
		\stackrel{\eqref{B-B-^H-adj}}{\cong} & \Hom_{\M^H}(V, \underrightarrow{\lim} (N_i\otimes B)).
	\end{align*}
	By the definition of direct limit, we have an exact sequence
	\[ \oplus (N_i \otimes B) \To \oplus (N_i \otimes B) \To \underrightarrow{\lim} (N_i\otimes B) \To 0\]
	in ${_B\!\M_B^H}$, which is split exact in $\M^H$ by Lemma \ref{Hopf-module-injective}.
	%	\begin{equation}\label{ex-seq-in-B-H-B}
		%		\bigoplus N_i \otimes B \To \bigoplus N_i \otimes B \To \underrightarrow{\lim} (N_i\otimes B) \To 0
		%	\end{equation}
	%	in ${_B\!\M_B^H}$.
	Hence we have a commutative diagram with exact rows:
	\[ \xymatrix{\Hom_{\M^H}(V, \bigoplus (N_i \otimes B)) \ar[r] \ar[d]^{\pi} & \Hom_{\M^H}(V, \bigoplus (N_i \otimes B)) \ar[r] \ar[d]^{\pi} & \Hom_{\M^H}(V, \underrightarrow{\lim} (N_i\otimes B)) \ar[r] \ar[d] & 0 \\
		\bigoplus \Hom_{\M^H}(V, N_i \otimes B) \ar[r] & \bigoplus \Hom_{\M^H}(V, N_i \otimes B) \ar[r] & \underrightarrow{\lim}\Hom_{\M^H}(V,  N_i\otimes B) \ar[r] & 0 \\}. \]
	Since $V$ is finite dimensional, the canonical map $\pi$ is bijective.
	This yields an isomorphism between $\Hom_{\M^H}(V, \underrightarrow{\lim} (N_i\otimes B))$ and $\underrightarrow{\lim}  \Hom_{\M^H}(V, N_i\otimes B)$.
	Therefore, we obtain the following isomorphisms,
	\begin{align*}
		\Hom_{L}( (B \otimes V \otimes B)^{co H}, \underrightarrow{\lim} N_i) \cong & \underrightarrow{\lim} \Hom_{\M^H}(V, N_i\otimes B) \\
		\stackrel{\eqref{B-B-^H-adj}}{\cong} & \underrightarrow{\lim} \Hom_{_B\!\M_B^H}(B \otimes V \otimes B, N_i\otimes B) \\
		\cong & \underrightarrow{\lim} \Hom_{L}((B \otimes V \otimes B)^{co H}, N_i).
	\end{align*}
	
	In conclusion, $(B \otimes V \otimes B)^{co H}$ is a finitely generated projective $L$-module.
\end{proof}

%\begin{lem}\label{B-otimes-V-otimes-B-fg-proj}
%	If $V$ is an $H$-comodule, then $B \otimes V \otimes B \cong (B \otimes B)^{\oplus \dim (V)}$ in ${_B\!\M_B^H}$.
%\end{lem}
%\begin{proof}
%	$V \otimes B \cong B^{\oplus \dim (V)}$
%\end{proof}

%Let $L$ and $H$ be two Hopf algebras with a bijective antipodes, and $B$ be an $L$-$H$-biGalois extension of $\kk$.
Recall that by \eqref{R-mod-str} and \eqref{L-comod-on-B}, for any left $L$-module $V$, $V \otimes B$ acquires a left $B^e$-bimodule structure defined by
$$(b' \otimes b'')(v \otimes b) = \sum\limits_{(b')} b'_{-1}v \otimes b'_0bb''.$$

\begin{lem}\label{fg-proj-H-B^e-mod}
	Let $V$ be a left $L$-module.
	
	(1) If $V$ is projective, then $V \otimes B$ is a projective $B^e$-module.
	
	(2) If $V$ is finitely generated, then $V \otimes B$ is also finitely generated as a $B^e$-module.
\end{lem}
\begin{proof}
	(1)
	If $V$ is projective, then it is a direct summand of some copies of $L$.
	Given that the Galois map $\beta': B \otimes B \to L \otimes B$ is an isomorphism of $B^e$-modules, $V \otimes B$ is also a projective $B^e$-module.
	
	(2)
	If $V$ is finitely generated, then there exists an integer $n$ with a surjective morphism $L^{\oplus n} \twoheadrightarrow V$.
	Consequently, $(B \otimes B)^{\oplus n} \cong L^{\oplus n}\otimes B \twoheadrightarrow V \otimes B$ is a surjective morphism of $B^e$-modules, confirming that $V \otimes B$ is finitely generated.
\end{proof}

\begin{prop}\label{H-B-FP-type}
	Let $H$ be a Hopf algebra with a bijective antipode, and $B$ be an $H$-Galois extension of $\kk$.
	For any $0 \leq n \leq \infty$, $H$ is of type $\mathbf{FP}_{n}$ if and only if $B$ is a $B^e$-module of type $\mathbf{FP}_{n}$.
%	\begin{enumerate}
%		\item For any $0 \leq n \leq \infty$, $H$ is of type $\mathbf{FP}_{n}$ if and only if $B$ is a $B^e$-module of type $\mathbf{FP}_{n}$.
%		\item If $H$ is of type $\mathbf{FP}$, then $B$ is a $B^e$-module of type $\mathbf{FP}$.
%	\end{enumerate}
\end{prop}
\begin{proof}
	$``\Rightarrow"$
	Given that $H$ is of type $\mathbf{FP}_{n}$, there exists a projective resolution $P_{\bullet}$ of $\kk_H$ with each $P_i$ is finitely generated for $i \leq n$.
	By Lemma \ref{fg-proj-H-B^e-mod}, tensoring $P_{\bullet}$ with $B$ yields a projective resolution of the $B^e$-module $B = \kk \otimes B$, where each $P_i\otimes B$ is finitely generated for $i \leq n$.
	Therefore, $B$ is a $B^e$-module of type $\mathbf{FP}_{n}$.
	
	$``\Leftarrow"$
	Since $B$ is a $B^e$-module of type $\mathbf{FP}_{n}$, Lemma \ref{type-FP-Hopf-modules} guarantees the existence of an exact sequence in the category ${_B\mathcal{M}_B^H}$ of the form
	\[ \cdots \To B \otimes V_i \otimes B \To \cdots \To B \otimes V_0 \otimes B \To B \To 0 \;\; \text{ in } {_B\mathcal{M}_B^H}, \]
	where each $V_i$ is an $H$-comodules for $i \geq 0$ and $V_i$ is finite-dimensional for $i \leq n$.
	Applying the functor of $H$-coinvariants $(-)^{co H}: {_B\mathcal{M}_B^H} \to {_L\M}$ yields an exact sequence 
	\[ \cdots \To (B \otimes V_i \otimes B)^{co H} \To \cdots \To (B \otimes V_0 \otimes B)^{co H} \To B^{co H} \To 0 \]
	in ${_L\M}$.
	Here, $B^{co H} = \kk$ is the trivial $L$-module.
	By Lemma \ref{B-otimes-V-otimes-B-fg-proj}, $(B \otimes V_i \otimes B)^{co H}$ is always finitely generated as an $L$-module for all $i \geq 0$.
	Thus, $H$ is of type $\mathbf{FP}_{n}$.
\end{proof}

\begin{thm}\label{mMT-preserve-FP_infty}
	Let $H$ and $L$ be two monoidally Morita--Takeuchi equivalent Hopf algebras with bijective antipodes.
	If $H$ is of type $\mathbf{FP}_{\infty}$, then so is $L$.
\end{thm}
\begin{proof}
	Let $B$ be an $L$-$H$-biGalois extension of $\kk$.
	Since the antipode $S$ of $L$ is bijective as we assumed, 
	%	$H^{op}$ is also a Hopf algebra with antipode $S^{-1}$.
	%	Let $R$ be a right $H$-comodule algebra. Thus $R^{op}$ is a right $H^{op}$-comodule algebra, and $R^{e}$ is a right $H^{e}$-comodule algebra.
	%	Notice that 
	$L^{cop}$ is also a Hopf algebra with opposite comultiplication antipode $S^{-1}$, and that $B$ is a right $L^{cop}$-Galois extension of $\kk$.
	Then the conclusion follows from Proposition \ref{H-B-FP-type}.
\end{proof}

\subsection{AS Gorenstein property under monoidal Morita--Takeuchi equivalence}

%In the remainder of the section, we develop an analog of Theorem \ref{mMTE-gld-thm} for the case when Hopf algebras have finite injective dimension rather than finite global dimension, the aim being to replace the AS regular condition with the AS Gorenstein condition.
In the subsequent portion of this section, our objective is to establish a parallel version of Theorem \ref{mMTE-gld-thm} for the scenario where Hopf algebras possess finite injective dimension, rather than finite global dimension. The ultimate goal is to substitute the AS regular condition with the AS Gorenstein condition.

According to \eqref{H-mod-on-M^A}, $B \otimes B$ serves as a right $B \otimes H$-module via the action defined by
\[(x \otimes y) (b \otimes h) = \sum \kappa^1(h) x b \otimes y \kappa^2(h).\]

\begin{lem}\label{B-otimes-B=B-otimes-H}
	Let $H$ be a Hopf algebra and $B$ be a right $H$-Galois extension of $\kk$.
	Then $B \otimes B \cong B \otimes H$ as right $B \otimes H$-modules. 
\end{lem}
\begin{proof}
%	Notice that $B \otimes B$ is a right $B \otimes H$-module via
%	\[(x \otimes y) (b \otimes h) = \sum \kappa^1(h) x b \otimes y \kappa^2(h).\]
	$B \otimes H$ is an $H$-comodule algebra, arising from the coaction $b \otimes h \mapsto \sum_{(h)} b \otimes h_1 \otimes h_2$.
	Verifying that $B \subseteq B \otimes H$ is a faithfully flat $H$-Galois extension is straightforward.
	Consequently, an equivalence functor $\big( - \otimes_B (B\otimes H), (-)^{co H} \big)$ exists between ${\M_B}$ and ${\M_{B \otimes H}^H}$.
	Computing $(\id_{B} \otimes \Delta)\big( (x \otimes y) (b \otimes h) \big)$ yields
	\begin{align*}
		= & \sum_{(y), (\kappa^2(h))} \kappa^1(h)x b \otimes y_0 \kappa^2(h)_0 \otimes y_1 \kappa^2(h)_1 \\
		\xlongequal[]{\eqref{kappa^1-kappa^2_0-kappa^2_1}} &  \sum_{(y), (\kappa^2(h))} \kappa^1(h_1)x b \otimes y_0 \kappa^2(h_1) \otimes y_1 h_2 \\
		= & \sum_{(y), (h)} \big( (x \otimes y_0) (b \otimes h_1) \big) \otimes y_1h_2,
	\end{align*}
	This establishes that $B \otimes B$ is a Hopf bimodule in the category $\M_{B \otimes H}^H$.
	Since $(B \otimes B)^{co H} = B \otimes B^{co H} \cong B$ as right $B$-modules, it follows that
	\[ B \otimes B \cong (B \otimes B)^{co H} \otimes_B (B\otimes H) = B \otimes H\]
	as right $B \otimes H$-modules, completing the argument.
\end{proof}

\begin{lem}\cite[Lemma 2.2]{Yu2016} \cite[Lemma 2.3.2]{WYZ2017} \label{Ext-lem}
	Let $H$ be a Hopf algebra of type $\mathbf{FP}_{\infty}$ with a bijective antipode, and $B$ be a $H$-Galois extension of $\kk$.
	Then 
	\[ B_B \otimes \Ext^i_{H^{op}}(\kk, H) \cong \Ext^i_{B^e}(B, B^e)\]
	as right $B$-modules for all $i$.
\end{lem}
\begin{proof}
%	Since the antipode $S$ of $L$ is bijective as we assumed, 
%	%	$H^{op}$ is also a Hopf algebra with antipode $S^{-1}$.
%	%	Let $R$ be a right $H$-comodule algebra. Thus $R^{op}$ is a right $H^{op}$-comodule algebra, and $R^{e}$ is a right $H^{e}$-comodule algebra.
%	%	Notice that 
%	$L^{cop}$ is also a Hopf algebra with antipode $S^{-1}$, and that $B$ is a right $L^{cop}$-Galois object.
%	Recall that $B$ is a right $L^{cop}$-Galois object.
%	By Lemma \ref,
%	\[ \Ext^i_{H^{op}}(\kk, B^e) \cong \Ext^i_{B^e}(B, B^e) \]
%	as $B^e$-modules.
	Since $\kk_H$ is of type $\mathbf{FP}_{\infty}$, we have
	\[ \Ext^i_{H^{op}}(\kk, B \otimes H) \cong (B \otimes H) \otimes_H \Ext^i_{H^{op}}(\kk, H) \cong B \otimes \Ext^i_{H^{op}}(\kk, H) \]
	for all $i$.
	On the other hand, by Lemma \ref{B-otimes-B=B-otimes-H}, $B^e \cong B \otimes H$ as right $B \otimes H$-modules.
	Consequently, applying Theorem \ref{H-Galois-spectral-sequence}, we obtain
	\[\Ext^i_{H^{op}}(\kk, B \otimes H) \cong \Ext^i_{H^{op}}(\kk, B^e) \cong \Ext^i_{H^{op}}(\kk, \Hom_{\kk^e}(\kk, B^e))\cong  \Ext^i_{B^e}(B, B^e) \]
	for all $i$.
%	So we get the conclusion.
	By combining the above isomorphisms, we arrive at the desired conclusion.
%	Since the antipode $S$ is bijective as we assumed, 
%		%	$H^{op}$ is also a Hopf algebra with antipode $S^{-1}$.
%		%	Let $R$ be a right $H$-comodule algebra. Thus $R^{op}$ is a right $H^{op}$-comodule algebra, and $R^{e}$ is a right $H^{e}$-comodule algebra.
%		%	Notice that 
%	$H^{cop}$ is also a Hopf algebra with antipode $S^{-1}$, and that $B$ is a left $H^{cop}$-Galois object.
%	By Lemma \ref{Ext_L-Ext_R-otimes-B},
%	\[ \Ext^i_{L}(\kk, B^e) \cong \Ext^i_{B^e}(B, B^e) \cong \Ext^i_{H}(\kk, B^e) \]
%	as $B^e$-modules.
%	Since $L \otimes B \cong B^e \cong B \otimes H$, it follows that 
\end{proof}

\begin{thm}\label{sCY-B-vs-AS-Gorenstein-H}
	Let $H$ be a Hopf algebra of type $\mathbf{FP}_{\infty}$ with a bijective antipode, and $B$ be a $H$-Galois extension of $\kk$.
	
	(1) \cite[Theorem 7]{Bic2022} If $\gld (H) = d < \infty$, then $\pd_{B^e}(B) = d$.
	
	(2) \cite[Theorem 2.5]{Yu2016} If $H$ is AS regular of dimension $d$, then $B$ is skew Calabi-Yau of dimension $d$.
	
	(3) If $\pd_{B^e}(B) = d$, then the injective dimension $\injdim(H) = d$.
	
	(4) If $B$ is skew Calabi-Yau of dimension $d$, then $H$ is AS Gorenstein of dimension $d$.
\end{thm}
\begin{proof}
	(1)
	Since $\kk_H$ is of type $\mathbf{FP}_{\infty}$, we can deduce from Lemma \ref{proj-dim-pseudo-coherent-mod} that $\Ext^d_{H^{op}}(\kk, H) \neq 0$.
	Hence $\Ext^d_{B^e}(B, B^e) \neq 0$ by Lemma \ref{Ext-lem}.
	Since $\pd_{B^e}(B) \leq \pd(\kk_H) = d$ by Corollary \ref{pdim-B-as-B-B-bimodule}, it follows that $\pd_{B^e}(B) = d$.
	
	(2)
	By Lemma \ref{Ext-lem}, 
	\[ \Ext^i_{B^e}(B, B^e) \cong \begin{cases}
		0, & i \neq d, \\
		B \otimes \Ext_{H^{op}}^d(\kk, H) \cong B, & i = d.
	\end{cases}\]
	as right $B$-modules.
	Similarly, we can prove that $\Ext^d_{B^e}(B, B^e) \cong B$ as left $B$-modules.
	Hence there exists an automorphism $\mu$ of $B$ such that $\Ext^d_{B^e}(B, B^e) \cong B_{\mu}$.
	So $B$ is skew Calabi-Yau of dimension $d$.
	
	(3)
%	Since $_{B^e}B$ is of type $\mathbf{FP}_{\infty}$ by Lemma \ref{H-B-FP-type}, it follows from Lemma \ref{proj-dim-pseudo-coherent-mod} that $\Ext^d_{B^e}(B, B^e) \neq 0$.
%	Hence $\Ext^d_{H^{op}}(\kk, H) \neq 0$ by Lemma \ref{Ext-lem}.
%	On the other hand, since  by Proposition \ref{id(L)-leq-pd(B)}, it follows that $\injdim(_LL) = d$.
	By Proposition \ref{H-B-FP-type}, $_{B^e}B$ is of type $\mathbf{FP}_{\infty}$, which ensures that $\Ext^d_{B^e}(B, B^e) \neq 0$ according to Lemma \ref{proj-dim-pseudo-coherent-mod}.
	Therefore, Lemma \ref{Ext-lem} implies that $\Ext^d_{H^{op}}(\kk, H) \neq 0$.
	On the other hand, since $\injdim(_LL) \leq d$ by Proposition \ref{id(L)-leq-pd(B)}, it follows that $\injdim(_LL) = d$.
	
	(4)
	If $B$ is skew Calabi-Yau of dimension $d$, then by Lemma \ref{Ext-lem} 
	\[ B \otimes \Ext^i_{H^{op}}(\kk, H) \cong \Ext^i_{B^e}(B, B^e) \cong \begin{cases}
		0, & i \neq d, \\
		B, & i = d.
	\end{cases}\]
	as right $B$-modules for all $i$.
	Hence $\Ext^d_{H^{op}}(\kk, H)$ is finite dimensional over $\kk$, see \cite[Corollary 1.2]{Lam1999} for example.
	Therefore, Lemma \ref{AS Gorenstein-lem} implies that $H$ is AS Gorenstein of dimension $d$.
\end{proof}

\begin{thm}\label{mMT-preserve-injdim}
	Let $H$ and $L$ be two monoidally Morita--Takeuchi equivalent Hopf algebras of type $\mathbf{FP}_{\infty}$.
	Suppose that the antipodes of $H$ and $L$ are both bijective.
	
	(1) If $H$ has finite global dimension $d$, then the injective dimension $\injdim(_LL) = d$.
	
	(2) If $H$ is AS regular of dimension $d$, then $L$ is AS Gorenstein of dimension $d$.
	
	(3) If $L$ and $H$ are both AS Gorenstein, then $\injdim (L)  = \injdim(H)$.
\end{thm}
\begin{proof}
	Observing that $B$ is a right $L^{cop}$-Galois extension of $\kk$, we can deduce (1) and (2) directly from Theorem \ref{sCY-B-vs-AS-Gorenstein-H}.
	
	(3)
	According to Lemma \ref{Ext-lem}, we have
	\[ \Ext^i_{L^{op}}(\kk, L) \neq 0 \text{ if and only if } \Ext^i_{B^e}(B, B^e) \neq 0 \text{ if and only if } \Ext^i_{H^{op}}(\kk, H) \neq 0.\]
	Since $L$ and $H$ are AS Gorenstein by our assumption, it follows that $\injdim (L)  = \injdim(H)$.
\end{proof}

\subsection{Monoidal Morita--Takeuchi equivalences don't preserve the global dimension.}

We conclude our discussion by presenting a counterexample to Question \ref{ques-0}.

\begin{ex}\label{ex}
	The Liu Hopf algebra $H:= L(n, \xi)$ \cite{Liu2009} is generated by $x, g, g^{-1}$ with the relations
	$$gg^{-1} = g^{-1}g = 1, \quad xg = \xi gx, \quad x^n = 1 - g^n,$$
	where $\xi \in \kk$ is an $n$-th primitive root of unity. The comultiplication, counit and antipode on $H$ are defined by
	$$\Delta(g) = g \otimes g, \; \Delta(x) = x \otimes 1 + g \otimes x, \; \vep(g) = 1, \; \vep(x) = 0, \; S(g) = g^{-1}, \; S(x) = -g^{-1}x.$$
	
	Let $B:= \kk\langle g, g^{-1}, t \rangle / \langle gg^{-1} - 1, g^{-1}g - 1, tg - \xi gt, t^n + g^n \rangle$.
	It serves as a right $H$-comodule algebra through the comodule map $\rho$ defined by
	$$\rho(g) = g \otimes g, \qquad \rho(t) = t \otimes 1 + g \otimes x.$$
	
	Clearly, the map $\gamma: H \to B$ given by $g^ix^j \mapsto g^it^j$ ($i \in \Z, j = 0, 1, \dots, n-1$), is a right $H$-comodule isomorphism.
	It is routine to check that $\gamma$ is convolution invertible, with convolution inverse $\gamma \circ S$.
	Consequently, $B$ is a cleft extension of $\kk$ by $H$, and hence a right $H$-Galois extension of $\kk$ by \cite[Theorem~8.2.4]{Mon1993}.
%	The bijectivity of the $H$-comodule map $\gamma: H \to B$ given by $g^ix^j \mapsto g^it^j$ ($i \in \Z, j = 0, 1, \dots, n$) implies its convolution invertibility with the inverse $\gamma S$. Consequently, $B$ is an $H$-cleft extension of $\kk$, and thus a right $H$-Galois extension of $\kk$ according to \cite[Theorem 8.2.4]{Mon1993}.
	
	Let $L$ be the Hopf algebra which is generated by $y, g, g^{-1}$ with the relations
	$$gg^{-1} = g^{-1}g = 1, \quad yg = \xi gy, \quad y^n = 0.$$
	The comultiplication, counit and antipode on $L$ are defined by
	$$\Delta(g) = g \otimes g, \; \Delta(y) = y \otimes 1 + g \otimes y, \; \vep(g) = 1, \; \vep(y) = 0, \; S(g) = g^{-1}, \; S(y) = -g^{-1}y.$$
	
	$B$ is also a left $L$-comodule algebra, defined through the comodule map $\delta$ with $\delta(g) = g \otimes g$ and $\delta(t) = y \otimes 1 + g \otimes t$. Furthermore, $B$ qualifies as a left $L$-Galois extension of $\kk$. Verifying that $B$ is an $L$-$H$-bicomodule is straightforward, leading to the conclusion that $B$ is an $L$-$H$-biGalois extension of $\kk$.
	
	The global dimension of $H$ is $1$, as stated in \cite[Lemma 2.4]{Liu2009}, whereas the global dimension of $L$ is infinite. This observation underscores that monoidal Morita--Takeuchi equivalences do not consistently preserve the global dimension. Nevertheless, since $H$ and $L$ are both noetherian affine PI Hopf algebras, they share the same injective dimension, as implied by Corollary \ref{PI-Hopf-alg-mMT-preserve-injdim-cor}.
\end{ex}

\section*{Acknowledgments}
The author thanks the referee for the valuable comments and suggestions.
The author is also grateful to Professor Xingting Wang for his useful conversations.
The research work was supported by the NSFC (project 12301052) and Fundamental Research Funds for the Central Universities (2022110884).

%who read the paper and made numerous helpful suggestions.

\bibliographystyle{siam}%{unsrt}
\bibliography{References}

\end{document}